\RequirePackage[l2tabu, orthodox]{nag}
\documentclass[a4paper,12pt]{amsart}
\usepackage{microtype}
\usepackage[charter]{mathdesign}

\usepackage[T1]{fontenc}
\usepackage{tikz, tikz-cd, stmaryrd, amsmath, amsthm, amssymb,
hyperref, bbm, mathtools, mathrsfs}
\newcommand{\upomega}{\boldsymbol{\omega}}
\newcommand{\upeta}{\boldsymbol{\eta}}
\newcommand{\dd}{\boldsymbol{d}}
\usepackage[shortlabels]{enumitem}
\usetikzlibrary{arrows}
\usetikzlibrary{positioning}
\usepackage[utf8x]{inputenc}
\newcommand{\bb}{\textbf}

\newcommand{\mc}{\mathcal}
\newcommand{\wh}{\widehat}

\newcommand{\mf}{\mathfrak}
\newcommand{\ms}{\mathscr}
\renewcommand{\AA}{\mathbb{A}}

\newcommand{\ZZ}{\mathbb{Z}}

\newcommand{\RR}{\mathbb{R}}
\newcommand{\PP}{\mathbb{P}}

\newcommand{\NN}{\mathbb{N}}
\newcommand{\FF}{\mathbb{F}}

\newcommand{\ddeg}{\textbf{deg}\,}

\DeclareMathOperator{\Ind}{Ind}

\DeclareMathOperator{\tr}{tr}

\DeclareMathOperator{\Span}{Span}

\DeclareMathOperator{\Gl}{Gl}

\DeclareMathOperator{\ord}{ord}
\DeclareMathOperator{\Spec}{Spec}

\DeclareMathOperator{\cha}{char}

\DeclareMathOperator{\Indec}{Indec}

\theoremstyle{plain}
\newtheorem{Theorem}{Theorem}[section]
\newtheorem*{mainthm}{Main Theorem}
\newtheorem{Remark}[Theorem]{Remark}
\newtheorem{Lemma}[Theorem]{Lemma}
\newtheorem{Corollary}[Theorem]{Corollary}

\newtheorem{Proposition}[Theorem]{Proposition}

\newtheorem{manualtheoreminner}{Theorem}
\newenvironment{manualtheorem}[1]{%
	\renewcommand\themanualtheoreminner{#1}%
	\manualtheoreminner
}{\endmanualtheoreminner}
\newtheorem{Question}[Theorem]{Question}

\theoremstyle{definition}

\numberwithin{equation}{section}
\begin{document}

\title[Indecomposable direct summands...]{Indecomposable direct summands\\ of cohomologies of curves}
\author[J. Garnek]{J\k{e}drzej Garnek}
\address{Institute of Mathematics of Polish Academy of Sciences\\ ul. \'{S}niadeckich 8, 00-656~Warszawa}
\address{Max Planck Institute for Mathematics\\ Vivatsgasse 7, 53111 Bonn, Germany}
\email{jgarnek@amu.edu.pl}
\urladdr{http://jgarnek.faculty.wmi.amu.edu.pl/}
\date{}  

\begin{abstract}
	Groups with a non-cyclic Sylow $p$-subgroup have too many representations over a field of characteristic~$p$ to describe them fully. A~natural question arises, whether the world of representations coming from algebraic varieties with a group action is as vast as the realm of all modular representations. In this article, we explore the possible ``building blocks'' (the indecomposable direct summands)
	of cohomologies of smooth projective curves with a group action. We show that usually there are infinitely many such possible summands. To prove this, we study a family of $\ZZ/p \times \ZZ/p$-covers and	describe the cohomologies of the members of this family completely.
\end{abstract}

\maketitle
\bibliographystyle{plain}

\section{Introduction} \label{sec:intro}
One of the fundamental problems of the representation theory is to determine the set $\Indec(\Lambda)$ of isomorphism classes of indecomposable $\Lambda$-modules of finite length for a fixed ring $\Lambda$. Most attention has been given to the case when
$\Lambda = k[G]$ is the group algebra of a finite group $G$ over a field $k$. 
In case when the characteristic of~$k$ is $p > 0$, it is known that $\Indec(k[G])$ is finite if and only if the $p$-Sylow subgroups of $G$ are cyclic (see e.g.~\cite{Higman}, \cite{Borevic_Faddeev}, \cite{Heller_Reiner_Reps_in_integers_I}).
Furthermore, when $p > 2$ and the $p$-Sylow subgroup of $G$ is not cyclic, the algebra $k[G]$ has a wild representation type.
In particular, classifying its modules is as difficult as the theory
of finite dimensional $k[X, Y]$-modules, which has been shown to contain undecidable statements, cf. \cite{Prest}. Hence, it is unrealistic to expect any complete description of $\Indec(k[G])$ in this case.\\

A rich source of modular representations is provided by algebraic geometry.
For instance, if $X$ is a smooth projective variety over a field of positive characteristic equipped with an action of a finite group,
its cohomologies are modular representations. It is natural to wonder whether the world of representations that arise in this way is as vast as the entire realm of modular representations.
In this article, we focus on the case of smooth projective curves and investigate the following question.
\begin{Question} \label{q:main}
	What can be said about the indecomposable direct summands of the cohomologies of curves with an action of a group?
\end{Question}
\noindent To be more precise, we would like to study the sets:
\begin{alignat*}{2}
	\Indec^{\star}(k[G]) &:= \{ M \in \Indec(k[G]) : M \textrm{ is a direct summand of } H^1_{\star}(X)\\
	&\, \textrm{ for a smooth projective curve } X/k \textrm{ with an action of } G  \},
\end{alignat*}
where $\star \in \{ Hdg, dR \}$ (recall that $H^1_{Hdg}(X) := H^0(X, \Omega_X) \oplus H^1(X, \mc O_X)$). For instance, we would like to know whether those sets are infinite.
The motivation for this question comes from the case when $k$ is an algebraically closed field of characteristic~$2$ and $G = \mathbb V_4$ (the Klein group). 
It turns out that in this case $\Indec^{Hdg}(k[G])$ is infinite (cf. \cite{Bleher_Camacho_Holomorphic_differentials}). On the other hand, surprisingly, assuming certain conjecture, $\Indec^{dR}(k[G])$ has only seven elements (see~\cite{Garnek_p_gp_covers_ii} and Remark~\ref{rem:Indec_HKG}). The goal of this article is to prove the following theorem, which explains the situation for other groups.
\begin{mainthm} \label{thm:main_thm}
	Keep the above notation. Let $k$ be an algebraically closed field of characteristic $p > 2$. Suppose that $G$ has a non-cyclic $p$-Sylow subgroup.
	Then the sets $\Indec^{Hdg}(k[G])$ and $\Indec^{dR}(k[G])$ are infinite.
\end{mainthm}
Using the classical methods, it seems quite challenging to establish results about $\Indec^{Hdg}(k[G])$ or $\Indec^{dR}(k[G])$
for a general class of groups. The novelty of our approach lies in 
decomposing the cohomologies of a Galois cover into cohomologies of Harbater--Katz--Gabber curves, using our previous results from~\cite{Garnek_p_gp_covers} and~\cite{Garnek_p_gp_covers_ii}. This allows one for instance to control the equivariant structure of the cohomologies of a cover that was created using patching theory.\\

We give now a sketch of the proof of Main Theorem. Note that the group~$G$ must contain a copy of~$H := \ZZ/p \times \ZZ/p$. In the first step of the proof, we give an explicit family of curves with an action of $H$ with infinitely many
possible indecomposable direct summands in the cohomologies. Namely, we consider the following family of $H$-covers of $\PP^1$ for any $\alpha \in k \setminus \FF_p$ and $m \in \ZZ_+$, $p \nmid m$:
\begin{equation} \label{eqn:family_Xma}
	X_{m, \alpha} : \quad y_0^p - y_0 = x^m, \quad y_1^p - y_1 = \alpha \cdot x^m.
\end{equation}
We describe the cohomologies of $X_{m, \alpha}$ completely. This allows us to prove that 
if $\alpha_1, \alpha_2 \in k \setminus \FF_p$, $m \ge 2$ and
\[
H^0(\Omega_{X_{m, \alpha_1}}) \cong H^0(\Omega_{X_{m, \alpha_2}})
\quad \textrm{ or } \quad H^1_{dR}(X_{m, \alpha_1}) \cong H^1_{dR}(X_{m, \alpha_2})
\]
as $k[H]$-modules, then $\alpha_1 = \alpha_2$ (see Theorem~\ref{thm:iso_of_holo_modules} and~\ref{thm:iso_of_dR_modules} for a more precise version of this statement). Note in particular, that the equivariant structure
of $H^1_{dR}(X_{m, \alpha})$ does not depend only on the ramification data (which was the case for $G = \ZZ/p$
and $G = \ZZ/2 \times \ZZ/2$).

In the second part of the proof of Main Theorem, we would like to
construct a curve~$Z$ with an action of~$G$ such that the cohomologies of~$Z$
contain the cohomologies of~$X_{m, \alpha}$ (for fixed $m$ and $\alpha$) as direct summands.
This construction would be straightforward if we allowed disconnected covers
(namely we could take the disjoint union $\coprod_{G/H} X_{m, \alpha} \to \PP^1$).
Instead, in order to ensure the connectedness, we construct a $G$-cover
$Z \to \PP^1$ that \emph{locally approximates} $\coprod_{G/H} X_{m, \alpha} \to \PP^1$,
using patching theory. The results of papers \cite{Garnek_p_gp_covers} and \cite{Garnek_p_gp_covers_ii} show that (under some additional assumptions) the cohomologies of~$Z$ considered as $k[H]$-modules contain the cohomologies of $X$ as a direct summand. \\

Our results leave several questions open. We mention two of them.
\begin{Question}
	Suppose that $\cha k = 2$ and that the $2$-Sylow subgroup $G$ is neither cyclic,
	nor isomorphic to $\ZZ/2 \times \ZZ/2$. Are the sets $\Indec^{Hdg}(k[G])$ and $\Indec^{dR}(k[G])$ infinite?
\end{Question}
\begin{Question}
	Fix a field $k$ of characteristic $p$ and a group $G$ with a non-cyclic $p$-Sylow subgroup. Are the sets $\{ \dim_k M : M \in \Indec^{\star}(k[G]) \}$ for $\star \in \{ Hdg, dR \}$ infinite?
\end{Question}
It seems likely that both questions can be answered using similar methods, but
considering other families of covers than $X_{m, \alpha}$.
\subsection*{Outline of the paper}
In Section~\ref{sec:notation} we discuss preliminaries on curves and modular representations. Also, we introduce some notation. Section~\ref{sec:family_of_covers}
concerns basic facts on the curves in the family $X_{m, \alpha}$. In Section~\ref{sec:holo_diffs} we study the modular representations associated to
the module of holomorphic differentials on $X_{m, \alpha}$. The next section yields similar results for the de Rham cohomology.
In the last section we prove Main Theorem, using the patching theory and the decomposition of cohomologies from~\cite{Garnek_p_gp_covers} and~\cite{Garnek_p_gp_covers_ii}.
\subsection*{Acknowledgements}
Part of this work was done in February and March 2024 during a stay at the Max-Planck
Institut f\"{u}r Mathematik (Bonn, Germany), whose hospitality and support are gratefully acknowledged. The author wishes to express his thanks to Wojciech Gajda, Piotr Achinger and his officemates from MPIM for many stimulating conversations.
The author was supported by the research grant SONATINA 6 "The de~Rham cohomology of $p$-group covers" UMO-2022/44/C/ST1/00033,
awarded by National Science Center, Poland.

\section{Notation and preliminaries} \label{sec:notation}
For any curve $X$ over a field~$k$, we denote by $g_X$ its arithmetic genus and by $k(X)$ its function field. If $P \in X(k)$, we write $\mc O_{X, P}$ and $\mf m_{X, P}$ for the local ring at~$P$ and its maximal ideal. We denote by $\ord_P : k(X) \to \ZZ \cup \{ \infty\}$ the valuation associated to~$P$. In the sequel, we abbreviate $\Omega_{X/k}$ to $\Omega_X$.
Suppose now that $\pi : X \to Y$ is a $G$-Galois cover of curves and $P \in X(k)$. Then there is a~left action of~$G$ on $X$ and a right action on $k(X)$. We denote
the image of $f \in k(X)$ under the action of $\sigma \in k[G]$ by $\sigma f$ or $\sigma \cdot f$.
Note that if $G$ acts on $X$ from left, this is a right action (but in our applications~)

We write also $e_{X/Y, P}$ for the ramification index at $P$ and $G_P$ for the stabilizer at $P$. Recall that the $i$th lower ramification group at~$P$ is defined as follows:
\[
	G_{P, i} := \{ \sigma \in G_P : \sigma f \equiv f \pmod{\mf m_{X, P}} \quad \forall_{f \in \mc O_{X, P}} \}.
\] 
Let also:
\begin{equation*}
	d_{X/Y, P} := \sum_{i \ge 0} (\# G_{P, i} - 1) \quad \textrm{ and } \quad 
	d_{X/Y, P}' := \sum_{i \ge 1} (\# G_{P, i} - 1).
\end{equation*}
Note that $d_{X/Y, P}$ is the exponent of the different of $k(X)/k(Y)$ at $P$. For a sheaf~$\mc F$ on~$X$ and $Q \in Y(k)$ we often
abbreviate $(\pi_* \mc F)_Q$ to $\mc F_Q$ and $(\pi_* \mc F)_Q \otimes \wh{\mc O}_{Y, Q}$ to $\wh{\mc F}_Q$. If $G$ is a~subgroup
of a group $G_1$ we write $\Ind^{G_1}_G X \to Y$ for the disconnected $G$-cover $\coprod_{G_1/G} X \to Y$.\\

Suppose that $G$ is an arbitrary group with a subgroup~$H$ and that
$M$ is a $k[G]$-module of finite dimension. In this case, we write $M^{\vee}$ for the
$k[G]$-module dual to~$M$ and $M|_H$ for the restriction of~$M$ to~$H$. Also,
we denote the augmentation ideal of~$G$ over~$k$ by~$I_G$. In what follows, we use several times the following standard fact, true under the assumption that $G$ is a $p$-group:
\begin{equation} \label{eqn:G-invariants_p_gp}
	\textrm{ if } M \neq 0, \textrm{ then } M^G \neq 0.
\end{equation}
From now on, we denote by $H$ the group $\ZZ/p \times \ZZ/p$, unless stated otherwise. Let $\sigma$ and $\tau$ be the generators of $H$. Write $\sigma_0 := \sigma - 1 \in k[H]$, $\tau_0 := \tau - 1 \in k[H]$.
Note that $\sigma_0^p = \tau_0^p = 0$ and that $I_H = \sigma_0 k[H] + \tau_0 k[H]$.
For any $k[H]$-module $M$ consider the following filtration:
\begin{align*}
	S_n(M) &:= \{ x \in M : I_H^i \cdot x = 0 \quad \forall_{i > n} \}\\
	&= \{ x \in M : \sigma_0^i \tau_0^j x = 0 \quad \forall_{i+j > n} \}.
\end{align*}
In particular, $S_{-1}(M) = \{ 0 \}$ and
\begin{equation*}
	S_{n+1}(M) := \sigma_0^{-1}(S_n(M)) \cap \tau_0^{-1}(S_n(M)).
\end{equation*}
This filtration induces the function $\ddeg : M \to \ZZ_{\ge -1}$ as follows:
\begin{align*}
	\ddeg v &:= \max \{ i : v \in S_i(M) \}.
\end{align*}
For any integer $n$, denote by $n^{(0)}, n^{(1)}, \ldots$ its digits in the $p$-adic expansion.
Recall that by Lucas's theorem (cf. \cite{Lucas_theorem}) for any $n, i \in \ZZ$:
\begin{equation} \label{eqn:lucas}
	{n \choose i} \equiv \prod_{j \ge 0} {n^{(j)} \choose i^{(j)}} \pmod{p}.
\end{equation}
Let also $s_p(n) := \sum_{j \ge 0} n^{(j)}$ be the sum of $p$-adic digits of~$n$.\\

Assume $p \nmid m$. For any $k$-vector space $V$ with a representation $\rho : \ZZ/m \to \Gl(V)$ we have the following decomposition into eigenspaces:
\begin{equation*}
	V \cong V(1) \oplus V(\zeta_m) \oplus \ldots \oplus V(\zeta_m^{m-1}),
\end{equation*}
where for $0 \le c \le m-1$ we write:
\[
V(\zeta_m^c) := \{ v \in V : \rho(1) v = \zeta_m^c \cdot v \}.
\]
Moreover, if $V$ is equipped with an action of a group $H$ commuting with~$\rho$,
the spaces $V(\zeta_m^c)$ become $k[H]$-modules.

\section{A family of $\ZZ/p \times \ZZ/p$-covers} \label{sec:family_of_covers}
The goal of this section is to describe the basic properties of the curves from the family~\eqref{eqn:family_Xma}. Fix $m \in \NN$, $p \nmid m$ and $\alpha \in k \setminus \FF_p$. Let $X := X_{m, \alpha}$. The group $H := \langle\sigma, \tau \rangle \cong \ZZ/p \times \ZZ/p$ acts on $X$ by the formulas
\begin{align*}
	(\sigma z_0, \sigma z_1, \sigma x) &:= (z_0 + 1, z_1, x),\\
	(\tau z_0, \tau z_1, \tau x) &:= (z_0, z_1+1, x).
\end{align*}
Observe that the curve $X$ admits also an action of $\ZZ/m$ given by
$(z_0, z_1, x) \mapsto (z_0, z_1, \zeta_m \cdot x)$ that commutes with the action of $H$.
We denote the cover $X \to X/H \cong \PP^1$ by~$\pi$. Define
also $P$ as the unique point of $X$ above $\infty$ and let $U := X \setminus \{ P \}$. Note that $P$ is the only ramified point of $\pi : X \to \PP^1$. Therefore $\pi$ is an \emph{Harbater--Katz--Gabber cover} (for $p$-groups this notion is equivalent to being a cover of the projective line ramified only at one point). The ramification groups are as follows:
\begin{equation*}
	H_{P, 0} = \ldots = H_{P, m} = H, \quad H_{P, m+1} = H_{P, m+2} = \ldots = 0
\end{equation*}
(cf. \cite[Section~3]{Wu_Scheidler_Ramification_groups_AS_extensions}). Therefore, the exponent of the different at $P$ equals:
\begin{equation*}
	d_{X/\PP^1, P} = (p^2 - 1) \cdot (m+1).
\end{equation*}
Thus, using the Riemann--Hurwitz formula (cf. \cite[Corollary~IV.2.4]{Hartshorne1977}) we see that 
\begin{equation*}
	g_X = \frac{1}{2}(p^2 - 1) \cdot (m-1).
\end{equation*}
Let $z := z_0 + \beta z_1 \in k(X)$, where $\beta := (-\alpha)^{-1/p}$. Note that:
\begin{equation} \label{eqn:group_action_z}
	\sigma z^n = \sum_{i = 0}^n {n \choose i} z^i \quad \textrm{ and }  \quad 
	\tau z^n = \sum_{i = 0}^n {n \choose i} \beta^{n-i} \cdot z^i.
\end{equation}
In the sequel we will also need valuations at~$P$ of selected functions on~$X$, as given in the table below.
\begin{table}[h]
	\centering
	\begin{tabular}{|c|c|}
		\hline
		\textbf{Function} & \textbf{Valuation at $P$} \\
		\hline
		$z_0$ & $-m \cdot p$ \\
		\hline
		$z_1$ & $-m \cdot p$ \\
		\hline
		$x$ & $-p^2$ \\
		\hline
		$dx$ & $(p^2 - 1)(m+1) - 2 p^2$ \\
		\hline
		$z$ & $-m$ \\
		\hline
	\end{tabular}
\end{table}

We justify now the given valuations. Let $\pi' : X' \to \PP^1$ be the $\ZZ/p$-cover given by the equation $z_0^p - z_0 = x^m$ and let $P' \in X'(k)$ be the unique point
above $\infty \in \PP^1$. Then $\ord_{P'}(z_0) = -m$ (see e.g. \cite[Lemma~4.2]{Garnek_equivariant}) and
\[
	\ord_P(z_0) = e_{X/X', P} \cdot \ord_{P'}(z_0) = -p \cdot m,
\]
By symmetry, $\ord_P(z_1) = -p \cdot m$. Analogously $\ord_P(x) = e_{X/\PP^1, P} \cdot \ord_{\infty}(x) = -p^2$. In order to calculate
$\ord_P(dx)$, we use \cite[Proposition~IV.2.2~(b)]{Hartshorne1977}:
\begin{align*}
	\ord_P(dx) &= e_{X/\PP^1, P} \cdot \ord_{\infty}(dx) + d_{X/\PP^1, P}\\
	&=(p^2 - 1)(m+1) - 2 p^2. 
\end{align*}
Finally:
\begin{align*}
	z^p = z_0^p - \alpha^{-1} \cdot z_1^p = (z_0 + x^m) - \alpha^{-1} \cdot (z_1 + \alpha \cdot x^m)
	= z_0 - \alpha^{-1} \cdot z_1.
\end{align*}
Thus:
\[
\ord_P(z) = \frac 1p \ord_P(z_0 - \alpha^{-1} \cdot z_1) \ge \frac 1p \min\{ \ord_P(z_0), \ord_P(z_1) \} = - m.
\]
Suppose to the contrary that $\ord_P(z) > -m$. Then:
\[
\ord_P(z_0 - \alpha^{-1} \cdot z_1), \quad \ord_P(z_0 + \beta \cdot z_1) > -m \cdot p
\]
and hence also
\[
\ord_P(z_1) = \ord_P((z_0 + \beta \cdot z_1) - (z_0 - \alpha^{-1} \cdot z_1)) > -m \cdot p.
\]
The contradiction shows that $\ord_P(z) = -m$.
\begin{Lemma} \label{lem:rr_spaces}
	For any $\delta \ge 0$, the set
	\[
	\{ z^i \cdot x^c : 0 \le i < p^2, \quad c \ge 0, \quad m \cdot i + p^2 \cdot c < \delta \}.
	\]
	is a basis of the $k$-vector space $H^0(X, \mc O_X((\delta-1) P))$.
\end{Lemma}
\begin{proof}
Consider $\mc H(P)$, the Weierstrass semigroup of $X$, defined as:
\[
\mc H(P) := \{ \delta \ge 0 : \ord_P(f) = - \delta \textrm{ for some } f \in H^0(U, \mc O_X) \}.
\]
Then $\mc H(P) = \ZZ_+ p^2 + \ZZ_+ m$ by \cite[Theorem~13]{Karanikolopoulos_Kontogeorgis_Autos}. Therefore, since $\ord_P(x^2) = -p^2$ and $\ord_P(z) = -m$, we obtain by \cite[Proposition~9]{Kontogeorgis_Tsouknidas_Cohomological_HKG}:
\[
H^0(X, \mc O_X((\delta-1) P)) = \Span_k \{ z^i \cdot x^c : 0 \le i, c, \quad m \cdot i + p^2 \cdot c < \delta \}. \qedhere
\]
\end{proof}
\noindent In the sequel we need also the following result.
\begin{Lemma} \label{lem:trace_of_z}
	We have $\tr_H(z^{p^2-1}) = (\beta^p - \beta)^{p-1}$.
\end{Lemma}
\begin{proof}
	By definition:
	\[
		\tr_H(z^{p^2-1}) = \sum_{i, j \in \FF_p} (z + i + j \cdot \beta)^{p^2-1}.
	\]
	Therefore it suffices to show the following identity in the ring $\FF_p[x, y]$:
	\begin{equation} \label{eqn:polynomial_identity}
		\sum_{i, j \in \FF_p} (x + i + j \cdot y)^{p^2-1} = (y^p - y)^{p-1}.
	\end{equation}
	Denote the left hand side of~\eqref{eqn:polynomial_identity} by $P(x, y)$.
	Write $P(x, y) = \sum_{n \ge 0} a_n(x) \cdot y^n$. Then:
	\[
		a_n(x) = {p^2 - 1 \choose n} \cdot \left(\sum_{i \in \FF_p} (x+i)^{p^2 - 1 - n} \right) \cdot \left( \sum_{j \in \FF_p}  j^n \right).
	\]
	Note that $a_{p^2 - 1}(x) = 0$, since the first sum in the product vanishes.
	If $p^2 - p < n < p^2 - 1$, then $a_n(x) = 0$, as $\sum_{j \in \FF_p} j^n = 0$, see \cite[Lemma~10.3]{IrelandRosen1990}. Moreover:
	\[
		\sum_{j \in \FF_p} j^{p^2 - p} = \sum_{i \in \FF_p} (x+i)^{p-1} = -1
	\]
	(see e.g. \cite[\S 3, Proposition 3]{Madden_Arithmetic_generalized})
	and ${p^2 - 1 \choose p^2 - p} \equiv 1 \pmod p$ by \eqref{eqn:lucas}. This implies that
	$a_{p^2 - p}(x) = 1$.
	To summarize, the degree of $P(x, y)$
	treated as a polynomial in $y$ equals $p^2 - p$ and the leading coefficient equals $1$. 
	Moreover, note that $P(x, y_0) = 0$ for any $y_0 \in \FF_p$. Indeed, if $y_0 \neq 0$, then by substituting $j' := j \cdot y_0$:
	\[
	P(x, y_0) = \sum_{i, j' \in \FF_p} (x + i + j')^{p^2-1}
	= 0,
	\]
	since every summand appears $p$ times. Similar reasoning shows that $P(x, 0) = 0$.
	This ends the proof, as $(y^p - y)^{p-1}$ is the only normed polynomial
	of degree $p^2 - p$ that vanishes at every element of $\FF_p$.
\end{proof}

\section{Holomorphic differentials} \label{sec:holo_diffs}
\noindent In this section we study the modular representations
arising from the family of $H$-covers given by~\eqref{eqn:family_Xma}.
Write
\begin{equation*}
	\dd(c) := p^2 - \left\lceil \frac{p^2 \cdot c + 1}{m} \right\rceil.
\end{equation*}
The following is the main result of this section.
\begin{Theorem} \label{thm:iso_of_holo_modules}
	Suppose that $m \in \ZZ_+$, $p \nmid m$.
	For any $\alpha \in k \setminus \FF_p$ and $1 \le c \le m-1$ the $k[H]$-module $H^0(\Omega_{X_{m, \alpha}})(\zeta_m^c)$ is indecomposable. Moreover, if $\alpha_1, \alpha_2 \in k \setminus \FF_p$, $\dd(c) \neq 0, 1, p^2 - 1$ and
	\begin{equation*}
		H^0(\Omega_{X_{m, \alpha_1}})(\zeta_m^c) \cong H^0(\Omega_{X_{m, \alpha_2}})(\zeta_m^c),
	\end{equation*}
	as $k[H]$-modules then $\alpha_1 = \alpha_2$.
\end{Theorem}
\begin{Remark}
	One checks that $\dd(c) \neq 0, 1, p^2 - 1$ holds if and only if
	\begin{equation*}
		\frac{m-1}{p^2} < c \le m - \frac{2m+1}{p^2}.
	\end{equation*}
	This easily implies that if $m \ge 2$ and $p \ge 3$ then there exists $1 \le c \le m-1$ such that $\dd(c) \neq 0, 1, p^2 - 1$.
\end{Remark}
\noindent We want now to translate Theorem~\ref{thm:iso_of_holo_modules} into
a result about a specific family of modular representations.
We fix $m \ge 1$, $\alpha \in k \setminus \FF_p$ and write
$X := X_{m, \alpha}$.
Note that the map:
\[
H^0(X, \mc O_X(\delta-1) P) \to H^0(X, \Omega_X), \qquad f \mapsto f \, dx
\]
(where $\delta := p^2 m - p^2 - m$) is an isomorphism. Therefore by Lemma~\ref{lem:rr_spaces}
the basis of $H^0(X, \Omega_X)$ is given by the forms:
\[
\omega_{i, c} := z^i \cdot x^{c-1} \, dx,
\]
where $i < p^2$ and $m \cdot i + p^2 \cdot c < m \cdot (p^2 - 1)$. Thus
the basis of $H^0(X, \Omega_X)(\zeta_m^c)$ is given by the forms
$\omega_{i, c}$, where
\[
i \in \mc I(c) := \{ i \in \ZZ : 0 \le i \le p^2 - 1, \quad m \cdot i + p^2 \cdot c < m \cdot (p^2 - 1) \}.
\]
(note in particular, that $\mc I(0) = \varnothing$). It follows easily that
the dimension of $H^0(X, \Omega_X)(\zeta_m^c)$ equals $\dd(c)$.\\

Observe that by~\eqref{eqn:group_action_z} the group action on the forms $\omega_{i, c}$ is given in the following way:
\begin{equation*}
	\sigma \omega_{n, c} = \sum_{i = 0}^n {n \choose i} \omega_{i, c}, \quad
	\tau \omega_{n, c} = \sum_{i = 0}^n {n \choose i} \beta^{n - i} \cdot \omega_{i, c}.
\end{equation*}
This allows us to give an abstract description of the $k[H]$-module $H^0(\Omega_{X_{m, \alpha}})(\zeta_m^c)$. For any positive integer $d \le p^2$ and $\beta \in k$, define 
$V_{d, \beta}$ to be the $d$-dimensional $k$-vector space spanned by the vectors $\upomega_i$, $0 \le i < d$. For simplicity, denote $\upomega_i = 0$ for any $i < 0$.
Consider the group action of $H$ on $V_{d, \beta}$ given as follows:
\begin{align*}
	\sigma \upomega_n = \sum_{i=0}^n {n \choose i} \upomega_i, \qquad \tau \upomega_n = \sum_{i=0}^n {n \choose i} \beta^{n-i} \cdot \upomega_i.
\end{align*}
One can easily prove by induction that for any $j \in \ZZ$:
\[
\sigma^j \cdot \upomega_n = \sum_{i=0}^n {n \choose i} j^{n-i} \cdot \upomega_i, \quad
\tau^j \cdot \upomega_n = \sum_{i=0}^n {n \choose i} (j \cdot \beta)^{n-i} \cdot \upomega_i.
\]
This implies that $\sigma$ and $\tau$ are of order $p$.
Also, one checks that $\sigma$ and~$\tau$ commute. Therefore $V_{d, \beta}$ is a well-defined $k[H]$-module. It is straightforward that the map
\begin{equation} \label{eqn:span_is_isom_to_Vd}
	\Span_k(1, z, \ldots z^{d-1}) \to V_{d, \beta}, \qquad z^i \mapsto \upomega_i
\end{equation}
(where $z$ is defined as in Section~\ref{sec:family_of_covers})
is an isomorphism of $k[H]$-modules. In particular, for any $1 \le c \le m-1$
we have the following isomorphism of $k[H]$-modules:
\begin{equation} \label{eqn:H0Omega_isom_Vd}
	H^0(\Omega_{X_{m, \alpha}})(\zeta_m^c) \to V_{\dd(c), \beta}, \quad \omega_{i, c} \mapsto \upomega_i.
\end{equation}
\begin{Lemma} \label{lem:properties_of_vd}
	Let $\beta \in k \setminus \FF_p$.
	\begin{enumerate}
		\item If $0 \le d_1 \le d_2 \le p^2$, then $V_{d_1, \beta}$ is a submodule of $V_{d_2, \beta}$.
		
		\item We have $V_{p^2, \beta} \cong k[H]$ and $V_{p^2-1, \beta} \cong I_H$. 
	\end{enumerate}
\end{Lemma}
\begin{proof}
	(1) Let $(\upomega_i)_{i < d_1}$ and $(\upomega_i')_{i < d_2}$ be
	the bases of $V_{d_1, \beta}$ and $V_{d_2, \beta}$ respectively. One checks easily that
	the map $\upomega_i \mapsto \upomega_i'$ for $0 \le i < d_1$ induces a monomorphism $V_{d_1, \beta} \hookrightarrow V_{d_2, \beta}$.
	
	(2) Consider the map
	\begin{align*}
		k[H] \to \Span_k(1, z, \ldots, z^{p^2 - 1}), \quad \sum_{h \in H} a_h h \mapsto \sum_{h \in H} a_h h \cdot z^{p^2 - 1}
	\end{align*}
	In order to show that this is an isomorphism, it suffices to show that $\tr_H(z^{p^2 - 1}) \neq 0$, see e.g.~\cite[Theorem~1]{Childs_Orzech_On_modular}. 
	This follows immediately from Lemma~\ref{lem:trace_of_z}.
	Thus by~\eqref{eqn:span_is_isom_to_Vd} $V_{p^2, \beta} \cong \Span_k(1, z, \ldots, z^{p^2 - 1}) \cong k[H]$.
	Note now that $\sigma_0 \upomega_{p^2 - 1}, \tau_0 \upomega_{p^2 - 1} \in V_{p^2-1, \beta}$ by Lemma~\ref{lem:sigma_0_v_first_terms}. Therefore
	$I_H \subset V_{p^2-1, \beta}$. The equality follows by comparing the dimensions. \qedhere
\end{proof}
\begin{Remark} \label{rmk:generic_Jordan_Vd}
	Even though the literature concerning the modular representations of $\ZZ/p \times \ZZ/p$
	is quite extensive, we didn't manage to find the family of representations $V_{d, \beta}$ in the literature. In recent years, special attention has been given to the class of modules of constant Jordan type (see e.g.~\cite{CFP_constant_Jordan}, \cite{Benson_survey}, \cite{Benson_reps_and_vbs}, \cite{Tanimoto_modular_reps}). It is easy to find examples demonstrating that the representations $V_{d, \beta}$ are generally not of constant Jordan type. The generic Jordan type of $V_{d, \beta}$ (in the sense of \cite{FPS_generic_and_maximal}) equals $d^{(1)} \cdot [p] + [d^{(0)}]$.
	We sketch now a proof of this fact. Recall that the generic type of $V_{d, \beta}$ is larger in the sense of the dominance order then the Jordan type of any
	element $\xi \in k[H]$ with $\xi^p = 0$ (cf. \cite{Wheeler}). On the other hand, $d^{(1)} \cdot [p] + [d^{(0)}]$ is the maximal partition in the sense of dominance order.
	Thus it suffices to show that the matrix of~$\sigma$ has the given Jordan type. One shows easily that the fixed vectors of $\sigma$ are given by:
	\[
	\sum_{i = 0}^n {n \choose i} (-1)^{n-i} \cdot \upomega_{p \cdot i + (n-i)}
	\]
	for $n = 0, \ldots, d^{(1)}$ (corresponding to the elements $(z^p - z)^n$). Thus $\dim_k V_{d, \beta}^{\sigma} = d^{(1)} + 1$, which is possible if and only if $\sigma$
	has $d^{(1)}$ blocks of size $p$ and one block of size $d^{(0)}$.
\end{Remark}
\begin{Remark} \label{rem:nice_decomposition_of_X_ma}
	The decomposition of $H^0(X, \Omega_X)$ into indecomposable direct summands takes a particularly nice form, when $m \equiv 1 \pmod{p^2}$. Indeed, if $m = p^2 \cdot m_0 + 1$ then:
	\[
	\dd(c) = p^2 - t \qquad \textrm{ for } (t-1) \cdot m_0 + 1 \le c \le t \cdot m_0.
	\]
	Hence, for $\beta := (-\alpha)^{-1/p}$:
	\[
		H^0(X, \Omega_X) \cong \bigoplus_{d = 1}^{p^2 - 1} V_{d, \beta}^{\oplus m_0}.
	\]
	In particular, all of the modules $V_{1, \beta}, \ldots, V_{p^2 - 1, \beta}$ appear as direct summands of $H^0(X, \Omega_X)$
	for $X := X_{p^2 + 1, \alpha}$, $\alpha := - \beta^{-p}$.
\end{Remark}

The isomorphism~\eqref{eqn:H0Omega_isom_Vd} easily implies that
Theorem~\ref{thm:iso_of_holo_modules} is equivalent to showing indecomposability of $V_{d, \beta}$ (cf. Corollary~\ref{cor:Vd_indecomposable}) along with the following result:
\begin{manualtheorem}{\ref{thm:iso_of_holo_modules}'}
	Suppose that $\beta_1, \beta_2 \in k \setminus \FF_p$ and $1 < d < p^2 - 1$. If $V_{d, \beta_1} \cong V_{d, \beta_2}$
	as $k[H]$-modules, then $\beta_1 = \beta_2$.
\end{manualtheorem}
The proof of Theorem~\ref{thm:iso_of_holo_modules}' will take the rest of this section.
The main ingredient is the following description of the degree function on $V_{d, \beta}$ (as defined in Section~\ref{sec:notation}). 
\begin{Proposition} \label{prop:homo_filtration}
	Suppose that $v = \sum_{i < d} c_i \cdot \upomega_i \in V_{d, \beta}$. Then
	\[
	\ddeg v = \max\{ s_p(i) : c_i \neq 0 \}.
	\]
\end{Proposition}
\noindent In order to prove Proposition~\ref{prop:homo_filtration} we need the following lemma.
\begin{Lemma} \label{lem:sigma_0_v_first_terms}
	For any $n$ there exist $a_0, a_1, \ldots, b_0, b_1, \ldots \in k$ such that:
	\begin{align*} 
		\sigma_0 \upomega_n &=
		n \cdot \upomega_{n-1} + n^{(1)} \cdot \upomega_{n-p} + \sum_{s_p(i) < s_p(n) - 1} a_i \cdot \upomega_i,\\
		\tau_0 \upomega_n &=
		n \cdot \beta \cdot \upomega_{n-1} + n^{(1)} \cdot \beta^p \cdot \upomega_{n-p} + \sum_{s_p(i) < s_p(n) - 1} b_i \cdot \upomega_i. 
	\end{align*}
\end{Lemma}
\begin{proof}
	Suppose that for some $i < n$ we have $s_p(i) \ge s_p(n) - 1$ and ${n \choose i} \not \equiv 0 \pmod p$. If $i^{(0)} > n^{(0)}$
	or $i^{(1)} > n^{(1)}$, then ${n \choose i} \equiv 0 \pmod p$ by~\eqref{eqn:lucas}.
	Therefore either $(i^{(0)}, i^{(1)}) = (n^{(0)} - 1, n^{(1)})$ (i.e. $i = n - 1$),
	or $(i^{(0)}, i^{(1)}) = (n^{(0)} - 1, n^{(1)})$ (i.e. $i = n - p$). Thus:
	\begin{align*}
		\tau_0 \upomega_n &= \sum_{i < n} {n \choose i} \beta^{n-i} \cdot \upomega_i\\
		&= {n \choose n-1} \beta \cdot \upomega_{n-1} + {n \choose n-p} \beta^p \cdot \upomega_{n-p}
		+ \sum_{s_p(i) < s_p(n) - 1} {n \choose i} \beta^{n-i} \cdot \upomega_i\\
		&= n \cdot \beta \cdot \upomega_{n-1} + n^{(1)} \cdot \beta^p \cdot \upomega_{n-p}
		+ \sum_{s_p(i) < s_p(n) - 1} {n \choose i} \beta^{n-i} \cdot \upomega_i
	\end{align*}
	(we used \eqref{eqn:lucas} again to show that ${n \choose n-p} \equiv n^{(1)} \pmod p$). One proves the second formula in the same manner.
\end{proof}
In the sequel we often use the following simple fact that holds for every $\beta \in k \setminus \FF_p$
and $a, b \in k$:
\begin{equation} \label{eqn:a=b=0}
	\textrm{ if } a+b = a \cdot \beta + b \cdot \beta^p = 0, \textrm{ then } a = b = 0.
\end{equation}

\begin{proof}[Proof of Proposition~\ref{prop:homo_filtration}]
	We prove by induction on $N$ that
	\begin{equation} \label{eqn:SN_induction}
		S_N(V_{d, \beta}) = \Span(\upomega_i : s_p(i) \le N).
	\end{equation}
	For $N = -1$ this is clear. Suppose now that~\eqref{eqn:SN_induction} holds for integers less than~$N$. Lemma~\ref{lem:sigma_0_v_first_terms} and induction hypothesis easily imply that
	for any~$i$ with $s_p(i) = N$ we have $\sigma_0 \upomega_i, \tau_0 \upomega_i \in S_{N-1}(V_{d, \beta})$, i.e. $\upomega_i \in S_N(V_{d, \beta})$.
	Write $v = \sum_{i < d} c_i \upomega_i \in S_N(V_{d, \beta})$. Let $M := \max \{ s_p(i) : c_i \neq 0 \}$.
	Suppose to the contrary that $M > N$ and fix any $n$ with $s_p(n) = M$. Then by Lemma~\ref{lem:sigma_0_v_first_terms}:
	\begin{align*}
		\tau_0 v = \sum_{s_p(i) = N} c_i \cdot (i \cdot \beta \cdot \upomega_{i-1} + i^{(1)} \cdot \beta^p \cdot \upomega_{i-p}) + \sum_{s_p(i) < M-1} b_i \cdot \upomega_i.
	\end{align*}
	Thus the coefficient of $\upomega_{n-1}$ in $\tau_0 v$ is equal to $\beta \cdot c_n \cdot n + \beta^p \cdot c_{n+p - 1} \cdot (n+p - 1)^{(1)}$. On the other hand, it
	equals $0$ by induction hypothesis. Therefore:
	\begin{equation} \label{eqn:coeff_upomega_tau0}
		\beta \cdot c_n \cdot n + \beta^p \cdot c_{n+p - 1} \cdot (n+p - 1)^{(1)} = 0.
	\end{equation}
	Similarly, by considering the coefficient of $\upomega_{n-1}$ in $\sigma_0 v$:
	\begin{equation} \label{eqn:coeff_upomega_sigma0}
		c_n \cdot n + c_{n+p - 1} \cdot (n+p - 1)^{(1)} = 0.
	\end{equation}
	The equations~\eqref{eqn:coeff_upomega_tau0} and~\eqref{eqn:coeff_upomega_sigma0} imply by~\eqref{eqn:a=b=0} that if $p \nmid n$ then $c_n = 0$.
	If $p | n$, then by considering the coefficients of $\upomega_{n - p}$ in
	$\sigma_0 v$ and $\tau_0 v$ we easily obtain:
	\begin{align*}
		c_n \cdot n^{(1)} + c_{n+p - 1} &= 0,\\
		\beta^p \cdot c_n \cdot n^{(1)} + \beta \cdot c_{n+p - 1} &= 0,
	\end{align*}
	which again implies that $c_n = 0$ by~\eqref{eqn:a=b=0}. Thus $c_n = 0$ for all $n$ with $s_p(n) = M$. This yields a contradiction and ends the proof.
\end{proof}
\begin{Corollary} \label{cor:Vd_indecomposable}
	The module $V_{d, \beta}$ is indecomposable.
\end{Corollary}
\begin{proof}
	This follows immediately from~\eqref{eqn:G-invariants_p_gp} by noting that $\dim_k V_{d, \beta}^H = \dim_k S_0(V_{d, \beta}) = 1$ by Proposition~\ref{prop:homo_filtration}.
\end{proof}
Suppose that $\Phi : V_{d, \beta_1} \to V_{d, \beta_2}$ is an isomorphism of $k[H]$-modules.
For clarity, we denote the basis of $V_{d, \beta_1}$ by $\upomega_n$ and the basis of $V_{d, \beta_2}$ by $\upomega_n'$. We divide the proof of Theorem~\ref{thm:iso_of_holo_modules}' into two separate cases:
\begin{itemize}
	\item \bb{Case I:} $1 < d < p^2 - p$.
	\item \bb{Case II:} $p^2 - p \le d < p^2 - 1$.
\end{itemize}
\subsection*{Proof of Theorem~\ref{thm:iso_of_holo_modules}' -- Case I} \label{subsec:pf_first_case}
Assume that $d < p^2 - p$. Then $D := d^{(1)} + 1$ is less then~$p$.
Let:
\[
\Phi(\upomega_{D-1}) = \sum_{i < d} a_i \cdot \upomega_i', \quad \Phi(\upomega_D) = \sum_{i < d} b_i \cdot \upomega_i'
\]
for some $a_i, b_i \in k$. Note that by Proposition~\ref{prop:homo_filtration} we have
$\ddeg(\Phi(\upomega_{D-1})) = \ddeg(\upomega_{D-1}) = D-1$ and thus
$a_i = 0$ for $s_p(i) > D - 1$. Similarly, $b_i = 0$ for $s_p(i) > D$ and
there exists~$0 \le i < d$ with $b_i \neq 0$ and $s_p(i) = D$.
We compute now $\tau_0 \Phi(\upomega_D)$ in two different ways.
On one hand, using Lemma~\ref{lem:sigma_0_v_first_terms}, for some $v_1 \in S_{D-2}(V_{d, \beta_1})$, $v_2 \in S_{D-2}(V_{d, \beta_2})$:
\begin{align} 
	\tau_0 \Phi(\upomega_D) &= \Phi(\tau_0 \upomega_D) = \Phi(D \cdot \beta_1 \cdot \upomega_{D-1} + v_1) \nonumber\\
	&= D \cdot \beta_1 \cdot \sum_{s_p(i) = D-1} a_i \upomega_i' + v_2. \label{eqn:Phi_tau_omega}
\end{align} 
On the other hand, using Lemma~\ref{lem:sigma_0_v_first_terms} one more time:
\begin{align} 
	\tau_0 \Phi(\upomega_D) &= \sum_{i < d} b_i \cdot \tau_0 \upomega_i' \nonumber\\
	&= \sum_{s_p(i) = D} b_i \cdot (i \cdot \beta_2 \cdot \upomega_{i-1}' + i^{(1)} \cdot \beta_2^p \cdot \upomega_{i-p}') + v_3, \label{eqn:tau_Phi_omega}
\end{align}
where $v_3 \in S_{D-2}(V_{d, \beta_2})$. Let $n := \max \{ i : b_i \neq 0, s_p(i) = D \}$. Then $p \nmid n$, since $p \cdot D > d$. Comparing the coefficients of $\upomega_{n-1}'$
in~\eqref{eqn:Phi_tau_omega} and in~\eqref{eqn:tau_Phi_omega} we obtain:
\begin{equation} \label{eqn:beta_a=beta_b}
	\beta_1 \cdot D \cdot a_{n-1} = \beta_2 \cdot n \cdot b_n.
\end{equation}
Analogously, by comparing the coefficients of $\upomega_{n-1}'$ in $\Phi(\sigma_0 \upomega_D) = \sigma_0 \Phi(\upomega_D)$ we obtain:
\begin{equation} \label{eqn:a=b}
	D \cdot a_{n-1} = n \cdot b_n.
\end{equation}
Finally, the equations~\eqref{eqn:beta_a=beta_b} and~\eqref{eqn:a=b} yield:
\begin{equation*}
	\beta_1 = \beta_1 \cdot \frac{D \cdot a_{n-1}}{n \cdot b_n} = \beta_2,
\end{equation*}
which ends the proof in this case.
\subsection*{Proof of Theorem~\ref{thm:iso_of_holo_modules}' -- Case II} \label{subsec:pf_2nd_case}

We assume now that $p^2 - p < d \le p^2 - 2$. Note that in this case, by Proposition~\ref{prop:homo_filtration}:
\[
	V_{d, \beta} = S_{2p-3}(V_{d, \beta}) = S_{2p-4}(V_{d, \beta}) \oplus \Span_k(\upomega_{p^2 - p - 1}).
\]
Let $\bb N(V_{d, \beta}) := \Span_k(v, w)$, where $v := \sigma_0^{p-2} \tau_0^{p-2} \upomega_{p^2 - p - 1}$ and $w := \sigma_0 v$.
\begin{Lemma} \label{lem:NV_is_V2}
	$\bb N(V_{d, \beta}) \cong V_{2, - \beta}$.
\end{Lemma}
\begin{proof}
	It suffices to prove that $v, w \neq 0$ and
	\[
		\sigma_0 v = w, \quad \tau_0 v = - \beta \cdot w.
	\]
	The first equality follows straight from the definition of $v$ and $w$.
	In order to prove the second equality, we work in the $V_{p^2-1, \beta}$.
	In this module, by Lemma~\ref{lem:sigma_0_v_first_terms}:
	\begin{align*}
		\sigma_0 \upomega_{p^2 - 1} &= - \upomega_{p^2 - 2} - \upomega_{p^2 - p - 1} + e_1,\\
		\tau_0 \upomega_{p^2 - 1} &= - \beta \cdot \upomega_{p^2 - 2} - \beta^p \cdot \upomega_{p^2 - p - 1} + e_2,
	\end{align*}
	where $\ddeg e_1, \ddeg e_2 \le 2p-4$. Therefore:
	\[
	- (\beta^p - \beta) \cdot \upomega_{p^2 - p - 1} = (\tau_0 - \beta \cdot \sigma_0) \cdot \upomega_{p^2 - 1} + e_3,
	\]
	where $\ddeg e_3 \le 2p-4$. This yields:
	\begin{align*}
		(\beta^p - \beta) \cdot (\tau_0 v + \beta w) &= \sigma_0^{p-2} \tau_0^{p-2} \cdot (\tau_0 + \beta \sigma_0) \cdot (\tau_0 - \beta \sigma_0) \cdot \upomega_{p^2 - 1}\\
		&= \sigma_0^{p-2} \tau_0^{p-2} \cdot (\tau_0^2 - \beta^2 \sigma_0^2) \cdot \upomega_{p^2 - 1} = 0,
	\end{align*}
	since $\sigma_0^p = \tau_0^p = 0$ and $\sigma_0^{p-1} \tau_0^{p-2} e_3 = \sigma_0^{p-2} \tau_0^{p-1} e_3 = 0$. Therefore, since $\beta \not \in \FF_p$, we obtain $\tau_0 v = - \beta w$. Finally, note that $w \neq 0$, since otherwise
	$\sigma_0 v = \tau_0 v = 0$ and $\ddeg \upomega_{p^2 - p - 1}$ would be equal
	to $2p-4$.
\end{proof}
We will show that $\Phi$ induces an isomorphism between $\bb N(V_{d, \beta_1})$ and
$\bb N(V_{d, \beta_2})$. Let $\bb N(V_{d, \beta_1}) = \Span_k(v, w)$ and $\bb N(V_{d, \beta_2}) = \Span_k(v', w')$,
where $v$, $w$, $v'$, $w'$ are defined as above. Since $S_{2p-3}(V_{d, \beta})/S_{2p-4}(V_{d, \beta}) = \Span_k(\upomega_{p^2 - p - 1})$,
we have 
\[
	\Phi(\upomega_{p^2 - p - 1}) = c_1 \cdot \upomega_{p^2 - p - 1}' + e',
\]
where $c_1 \in k^{\times}$ and $e' \in S_{2p-4}(V_{d, \beta_2})$.	Note that $\sigma_0^{p-1} \tau_0^{p-2} e' = \sigma_0^{p-2} \tau_0^{p-1} e' = 0$. Hence $\sigma_0^{p-2} \tau_0^{p-2} e' \in V_{d, \beta_2}^H = \Span_k(w')$ and $\sigma_0^{p-2} \tau_0^{p-2} e' = c_2 \cdot w'$ for some $c_2 \in k$. Therefore:
\[
	\Phi(v) = \sigma_0^{p-2} \tau_0^{p-2} (c_1 \cdot \upomega_{p^2 - p - 1}' + e')
	= c_1 \cdot v' + c_2 \cdot w'.
\]
Similarly, $\Phi(w) = c_1 \cdot w'$. Hence:
\begin{align*}
	\Phi(\bb N(V_{d, \beta_1})) = \Span_k(c_1 \cdot v' + c_2 \cdot w', c_1 \cdot w')
	= \Span_k(v', w') = \bb N(V_{d, \beta_2}).
\end{align*}
By Lemma~\ref{lem:NV_is_V2} we obtain $V_{2, - \beta_1} \cong V_{2, - \beta_2}$. Therefore
(by Case I of the proof) $\beta_1 = \beta_2$. This ends the proof.

\section{The de Rham cohomology}
In this section we prove the following analogue of Theorem~\ref{thm:iso_of_holo_modules}
for the de Rham cohomology.
\begin{Theorem} \label{thm:iso_of_dR_modules}
	Suppose that $m \in \ZZ_+$, $p \nmid m$. For any $\alpha \in k \setminus \FF_p$ and $1 \le c \le m-1$ the $k[H]$-module $H^1_{dR}(X)(\zeta_m^c)$ is indecomposable of dimension $p^2 - 1$. Moreover, if $\alpha_1, \alpha_2 \in k \setminus \FF_p$, $\lfloor \frac{m}{p} \rfloor \le c_1, c_2 \le m - \lfloor \frac{m}{p} \rfloor$ and
	\begin{equation*}
		H^1_{dR}(X_{m, \alpha_1})(\zeta_m^{-c_1}) \cong H^1_{dR}(X_{m, \alpha_2})(\zeta_m^{-c_2}),
	\end{equation*}
	as $k[H]$-modules, then $\lfloor \frac{\dd(c_1)}{p} \rfloor = \lfloor \frac{\dd(c_2)}{p} \rfloor$
	and $\alpha_1 = \alpha_2$.
\end{Theorem}
Again, let $X = X_{m, \alpha}$ for fixed $p \nmid m$ and $\alpha \in k \setminus \FF_p$.
We describe now the basis of the Zariski cohomology of the structure sheaf of $X$. Recall that
using \v{C}ech cohomology, this cohomology might be computed as:
\[
	H^1(X, \mc O_X) \cong \frac{\mc O_X(U \cap V)}{\mc O_X(U) + \mc O_X(V)}
\]
where $V := X \setminus \pi^{-1}(0)$. We denote the image of a function $f \in \mc O_X(U \cap V)$ in $H^1(X, \mc O_X)$ by $[f]$. 
\begin{Proposition} \label{prop:coh_of_str_sheaf}
	The basis of $H^1(X, \mc O_X)(\zeta_m^{-c})$ is given by $[z^i/x^c]$, where
	\[
		i \in \mc J(c) := \left \{ i \in \ZZ : 0 \le i \le p^2 - 1, \quad \frac{p^2 \cdot c}{m} < i \right\}.
	\]
\end{Proposition}
\begin{proof}
	Note that the elements $[z^i/x^c]$, $i \in \mc J(c)$ are linearly independent in $H^1(X, \mc O_X)(\zeta_m^{-c})$. Indeed, suppose that for some scalars $a_i \in k$ not all equal to
	zero:
	\[
		\sum_{i \in \mc J(c)} a_i \cdot [z^i/x^c] = 0, \quad \textrm{i.e. } \sum_{i \in \mc J(c)} a_i \cdot z^i/x^c \in \mc O_X(U) + \mc O_X(V).
	\]
	Since $\mc O_X(U) = k[x, z]$, there would exist $b_{ij} \in k$ such that
	\[
		f := \sum_{i \in \mc J(c)} a_i \cdot z^i/x^c + \sum_{i = 0}^{p^2 - 1} \sum_{j \ge 0} b_{ij} z^i \cdot x^j \in \mc O_{X, P}.
	\]
	Note however, that the functions in the set
	\[
		\{ z^i \cdot x^j : 0 \le i \le p^2 - 1, \quad j \in \ZZ \}
	\]
	have pairwise different valuations at~$P$. Indeed, if $\ord_P(z^{i_1} \cdot x^{j_1}) = \ord_P(z^{i_2} \cdot x^{j_2})$, then
	\[
		p^2 \cdot (j_1 - j_2) = m \cdot (i_2 - i_1).
	\]
	Thus $p^2 | i_2 - i_1$, which is possible only if $i_1 = i_2$ and $j_1 = j_2$.
	Moreover, $\ord_P(z^i/x^c) < 0$ for $i \in \mc J(c)$. Thus $f \in \mc O_{X, P}$
	would be sum of functions with poles of distinct orders at~$P$. The contradiction shows that
	the listed elements are linearly independent. Moreover, the map
	\[
		i \mapsto p^2 - 1 - i
	\]
	is a bijection between $\mc I(c)$ and $\mc J(c)$. Note that the Serre's duality
	$H^0(X, \Omega_X)^{\vee} \cong H^1(X, \mc O_X)$ (cf. \cite[Corollary III.7.7]{Hartshorne1977}) induces an isomorphism between $H^0(X, \Omega_X)(\zeta_m^c)^{\vee}$ and $H^1(X, \mc O_X)(\zeta_m^{-c})$. Therefore:
	\begin{align*}
		\dim_k \Span_k([z^i/x^c] : i \in \mc J(c)) &= \# \mc J(c)
		= \# \mc I(c)\\
		&= \dim_k H^0(X, \Omega_X)(\zeta_m^c)\\
		&= \dim_k H^1(X, \mc O_X)(\zeta_m^{-c}).
	\end{align*}
	Therefore $\Span_k([z^i/x^c] : i \in \mc J(c)) = H^1(X, \mc O_X)(\zeta_m^{-c})$, which ends the proof.
\end{proof}
We describe now the de Rham cohomology of $X$. Using \v{C}ech cohomology for the cover $(U, V)$ we
may identify $H^1_{dR}(X)$ with the space $Z^1(\Omega_X^{\bullet})/B^1(\Omega_X^{\bullet})$, where:
\begin{align*}
	Z^1(\Omega_X^{\bullet}) &:= \{ (\omega_0, h, \omega_{\infty}) \in \Omega_X(U) \times \mc O_X(U \cap V) \times \Omega_X(U): \omega_0 - dh = \omega_{\infty} \},\\
	B^1(\Omega_X^{\bullet}) &:= \{ (dh_0, h_0 + h_{\infty}, dh_{\infty}) : h_0 \in \mc O_X(U), h_{\infty} \in \mc O_X(V)  \}
\end{align*}
(see e.g. \cite[Section~2]{KockTait2018}). In the sequel we need also the Hodge--de Rham exact sequence:
\begin{equation} \label{eqn:hdr}
	0 \to H^0(X, \Omega_X) \stackrel{\iota}{\longrightarrow} H^1_{dR}(X) \stackrel{\phi}{\longrightarrow} H^1(X, \mc O_X) \to 0.
\end{equation}
Recall that the maps $\iota$ and $\phi$ are defined by:
\begin{align*}
	\iota(\omega) = (\omega, 0, \omega) \quad \textrm{ and } \quad
	\phi((\omega_0, h, \omega_{\infty})) = [h].
\end{align*}
We often abuse the notation by writing $\omega$ instead of $\iota(\omega)$, when the context is clear. By taking the eigenspaces in the exact sequence~\eqref{eqn:hdr} one obtains the following version of the Hodge--de Rham exact sequence:
\begin{equation} \label{eqn:hdr_zeta_m}
	0 \to H^0(X, \Omega_X)(\zeta_m^c) \stackrel{\iota}{\longrightarrow} H^1_{dR}(X)(\zeta_m^c) \stackrel{\phi}{\longrightarrow} H^1(X, \mc O_X)(\zeta_m^c) \to 0.
\end{equation}
\begin{Lemma}
	For any $1 \le c \le m-1$, $0 \le i \le p^2 - 1$, the element 
	\[
	\left( \frac{d(z^i)}{x^c},
	\frac{z^i}{x^c}, c \cdot \frac{z^i \, dx}{x^{c+1}} \right) \in \Omega_{k(X)} \times k(X) \times \Omega_{k(X)}
	\]
	belongs to $Z^1(\Omega_X^{\bullet})$.
\end{Lemma}
\begin{proof}
	Firstly, note that:
	\begin{align*}
		\frac{d(z^i)}{x^c} = \frac{i \cdot z^{i-1} \cdot dz}{x^c}
		&= \frac{i \cdot z^{i-1} \cdot (dz_0 + \beta \cdot dz_1)}{x^c}\\
		&= -m \cdot i \cdot (1 + \alpha \cdot \beta) \cdot z^{i-1} x^{m-1-c} \, dx
		\in \Omega_X(U).
	\end{align*}
	Moreover, the form $\frac{z^i \, dx}{x^{c+1}}$ is regular on $U \cap V = V \setminus \{ P \}$ and
	\begin{align*}
		\ord_P \left(\frac{z^i \, dx}{x^{c+1}} \right) &= (c+1) \cdot p^2 - i \cdot m + (p^2-1) \cdot (m+1) - 2p^2\\
		&\ge 2 p^2 - (p^2 - 1) \cdot m + (p^2-1) \cdot (m+1) - 2p^2\\
		&= p^2 - 1.
	\end{align*}
	Therefore $\frac{z^i \, dx}{x^{c+1}} \in \Omega_X(V)$. It is straightforward that
	$z^i/x^c \in \mc O_X(U \cap V)$ and that
	\[
		d(z^i/x^c) = \frac{d(z^i)}{x^c} - c \cdot \frac{z^i \, dx}{x^{c+1}}. \qedhere
	\]
\end{proof}
\noindent We denote the class in $H^1_{dR}(X)$ corresponding to $( \frac{d(z^i)}{x^c},
\frac{z^i}{x^c}, c \cdot \frac{z^i \, dx}{x^{c+1}})$ by $\eta_{i, c}$. 
\begin{Corollary} \label{cor:basis_of_h1dr}
	The basis of $H^1_{dR}(X)(\zeta_m^{-c})$ is given by $(\omega_{i, m-c})_{i \in \mc I(m - c)}$ and $(\eta_{i, c})_{i \in \mc J(c)}$.
\end{Corollary}
\begin{proof}
	Recall that the forms $\omega_{i, m-c}$ for $i \in \mc I(m-c)$
	form a basis of $H^0(\Omega_X)(\zeta_m^{-c})$ (see Section~\ref{sec:holo_diffs}). Moreover, the classes $\phi(\eta_{i,c})$ for $i \in \mc J(c)$ form a basis
	of $H^1(\mc O_X)(\zeta_m^{-c})$ by Proposition~\ref{prop:coh_of_str_sheaf}. Therefore the proof follows from the exact sequence~\eqref{eqn:hdr_zeta_m}.
\end{proof}
\begin{Lemma} \label{lem:eta_equals_omega}
If $i \not \in \mc J(c)$ and $1 \le c \le m$, then:
\[
\eta_{i, c} = - i \cdot \gamma \cdot \omega_{i-1, m-c},
\]
where $\gamma := m \cdot (1 + \alpha \cdot \beta) \in k^{\times}$.
\end{Lemma}
\begin{proof}
Note that if $i \not \in \mc J(c)$ then $z^i/x^c \in \mc O_X(V)$. This yields the following equality in $H^1_{dR}(X)$:
\begin{align*}
	\eta_{i, c} &= \left( \frac{d(z^i)}{x^c}, 0, c \cdot \frac{z^i \, dx}{x^{c+1}} - d \left(\frac{z^i}{x^c} \right) \right)
	\\
	&= \left( \frac{d(z^i)}{x^c}, 0, \frac{d(z^i)}{x^c} \right)\\
	&= -i \cdot m \cdot (1 + \alpha \cdot \beta) \cdot (z^{i-1} x^{m-1-c} \, dx, 0, z^{i-1} x^{m-1-c} \, dx)\\
	&= - i \cdot m \cdot (1 + \alpha \cdot \beta) \cdot \omega_{i-1, m-c}. \qedhere
\end{align*}
\end{proof}
\noindent The equations~\eqref{eqn:group_action_z} allow to easily obtain the group action
on the considered cocycles:
\begin{equation} \label{eqn:gp_action_etas}
	\sigma \eta_{n, c} = \sum_{i = 0}^n {n \choose i} \eta_{i, c}, \quad
	\tau \eta_{n, c} = \sum_{i = 0}^n {n \choose i} \beta^{n - i} \cdot \eta_{i, c}.
\end{equation}
This will allow us to give an abstract definition of the $k[H]$-module associated to $H^1_{dR}(X)(\zeta_m^{-c})$. 
For any $0 \le d \le p^2$ consider the following $k$-linear subspace of $V_{p^2, \beta} \oplus V_{d, \beta}$:
\[
	K_{\beta}^{(d)} := \Span_k ( (\upomega_i, 0) + i \cdot (0, \upomega_{i-1}) : 0 \le i \le d).
\]
Note that $K_{\beta}^{(d)}$ is a $k[H]$-submodule of $V_{p^2, \beta} \oplus V_{d, \beta}$.
Indeed, for any $0 \le n < d$:
\begin{align*}
	\tau ((\upomega_n, 0) + n \cdot (0, \upomega_{n-1})) &= \sum_{i = 0}^n \left( {n \choose i} \cdot \beta^{n-i} \cdot (\upomega_i, 0) + n \cdot {n-1 \choose i-1} \cdot \beta^{n-i} \cdot (0, \upomega_{i-1}) \right)\\
	&= \sum_{i = 0}^n {n \choose i} \cdot \beta^{n-i} \cdot \left((\upomega_i, 0) + i \cdot (0, \upomega_{i-1}) \right) \in K_{\beta}^{(d)}.
\end{align*}
Similarly one shows that $\sigma ((\upomega_n, 0) + n \cdot (0, \upomega_{n-1})) \in K_{\beta}^{(d)}$. Therefore
the quotient space:
\begin{equation*}
	V_{dR, \beta}^{(d)} := \frac{V_{p^2, \beta} \oplus V_{d, \beta}}
	{K_{\beta}^{(d)}}.
\end{equation*}
is a $k[H]$-module of dimension $p^2 - 1$ over~$k$. Note that the map $V_{d, \beta} \to V_{dR, \beta}^{(d)}$, $\upomega_n \mapsto (0, \upomega_n)$
is a monomorphism of $k[H]$-modules. Therefore by abuse of notation we may write $\upomega_n$ for the image of 
$(0, \upomega_n)$ in $V_{dR, \beta}^{(d)}$. Also, denote the image of $(\upomega_n, 0)$ by $\upeta_n$.
Corollary~\ref{cor:basis_of_h1dr}, Lemma~\ref{lem:eta_equals_omega} and formulas \eqref{eqn:gp_action_etas} easily imply that for $d := \dd(m-c)$ and $\beta := (-\alpha)^{-1/p}$ the map
\begin{align*}
	V_{dR, \beta}^{(d)} &\to H^1_{dR}(X)(\zeta_m^{-c}),\\
	\upomega_i &\mapsto  - i \cdot \gamma \cdot \omega_{i, m-c} &&\textrm{ for } i \in \mc I(m-c),\\
	\upeta_i &\mapsto \eta_{i, c} &&\textrm{ for } i \in \mc J(c)
\end{align*}
is an isomorphism of $k[H]$-modules.
\begin{Lemma} \label{lem:properties_of_v_dR}
	\begin{enumerate}
		\item[]
		\item If $d_1^{(1)} = d_2^{(1)}$,
		then $V_{dR, \beta}^{(d_1)} \cong V_{dR, \beta}^{(d_2)}$.
		
		\item Dual of the $k[H]$-module $V_{dR, \beta}^{(d)}$ is isomorphic to
		$V_{dR, \beta}^{(p^2 - 1 - d)}$.
		
		\item If $d < p$, then $V^{(d)}_{dR, \beta} \cong I_H^{\vee}$.
		Moreover, if $d \ge p^2 - p$, then $V^{(d)}_{dR, \beta} \cong I_H$.  
	\end{enumerate}
\end{Lemma}
\begin{proof}
	(1) Suppose that $d_1 > d_2$ and $d_1^{(1)} = d_2^{(1)}$. Denote the elements of the first module by $\upomega_i, \upeta_i$ and of the second by $\upomega_i', \upeta_i'$. Note that by assumption, the interval $[d_2 + 1, d_1]$ does not contain any number divisible by $p$. One may easily check that the map:
	\begin{align*}
		\Phi : V_{dR, \beta}^{(d_1)} &\to V_{dR, \beta}^{(d_2)},&\\
		\upomega_i &\mapsto \upomega_i' &\textrm{ for } i < d_2,\\
		\upomega_i &\mapsto \frac 1{i+1} \upeta_{i+1}' &\textrm{ for } d_2 \le i < d_1,\\
		\upeta_i &\mapsto \upeta_i' &\textrm{ for } i > d_2,
	\end{align*}
	yields an isomorphism of $k[H]$-modules. 
		
	(2) Let $X := X_{p^2 + 1, \alpha}$, where $\alpha := - \beta^{-p}$.
	By Remark~\ref{rem:nice_decomposition_of_X_ma} for any $0 \le d \le p^2$ we have $V_{dR, \beta}^{(d)} \cong H^1_{dR}(X)(\zeta_m^{d+1})$.
	Moreover, the cup product on $H^1_{dR}(X)$ 
	yields a duality between $H^1_{dR}(X)(\zeta_{p^2 + 1}^c)$ and $H^1_{dR}(X)(\zeta_{p^2 + 1}^{-c})$ for any $0 \le c \le p^2$.
	Therefore:
	\begin{align*}
		(V_{dR, \beta}^{(d)})^{\vee} = H^1_{dR}(X)(\zeta_{p^2 + 1}^{d+1})^{\vee} = H^1_{dR}(X)(\zeta_{p^2 + 1}^{p^2 - d})
		= V_{dR, \beta}^{(p^2 - 1 - d)}.
	\end{align*} 

	(3) Let $d < p$. Then, by Lemma~\ref{lem:properties_of_v_dR}~(2) and Lemma~\ref{lem:properties_of_vd}~(1):
	\[
		V^{(d)}_{dR, \beta} \cong V^{(0)}_{dR, \beta} \cong V_{p^2, \beta}/\Span(\upomega_0) \cong k[H]/k \cong I_H^{\vee}
	\]
	By Lemma~\ref{lem:properties_of_v_dR}~(2) this immediately implies that
	$V^{(d)}_{dR, \beta} \cong I_H$ for $d \ge p^2 - p$.
\end{proof}
\begin{Remark} 
	The generic Jordan type of $V_{dR, \beta}^{(d)}$ is $(p-1) \cdot [p] + [p-1]$ (cf. Remark~\ref{rmk:generic_Jordan_Vd}). Indeed, by Lemma~\ref{lem:properties_of_v_dR}~(1) we can assume without loss of generality that $p|d$. One can prove that the fixed space of
	$\sigma$ is generated by the vectors:
	\[
		\sum_{i = 0}^n {n \choose i} (-1)^{n-i} \cdot \upeta_{p \cdot i + (n-i)}
	\]
	for $n = 1, \ldots, p-1$ and $\upeta_{d + p}$. Thus the Jordan form of the operator $\sigma$ has $p$-blocks, which is possible only if it has $(p-1)$ blocks
	of size~$p$ and one block of size $(p-1)$. 
\end{Remark}
By the above discussion, Theorem~\ref{thm:iso_of_dR_modules}
becomes equivalent to showing indecomposability of $V_{dR, \beta}^{(d)}$ (cf. Lemma~\ref{lem:vdR_indecomposable}) along with the following statement.
\begin{manualtheorem}{\ref{thm:iso_of_dR_modules}'}
	Suppose that $p \le d_1, d_2 < p^2 - p$ and $\beta_1, \beta_2 \in k \setminus \FF_p$.
	If $V_{dR, \beta_1}^{(d_1)} \cong V_{dR, \beta_2}^{(d_2)}$
	as $k[H]$-modules, then $d_1^{(1)} = d_2^{(1)}$ and $\beta_1 = \beta_2$.
\end{manualtheorem}
Note that the basis of $V_{dR, \beta}^{(d)}$ is given by:
\[
	\{ \upeta_i : 1 \le i \le p^2 - 1, \, p \nmid i \textrm{ or } i > d \} \cup \{ \upomega_i : 0 \le i < d, i \equiv -1 \pmod p \}.
\]
Thus any $v \in V_{dR, \beta}^{(d)}$ might be uniquely written in the form 
\begin{equation} \label{eqn:v_in_basis}
	v = \sum_{i \equiv -1 \pmod p} a_i \upomega_i + \sum_i b_i \upeta_i,
\end{equation}
where $a_i, b_i \in k$ and the indices satisfy the above constraints.
Define the function $\ddeg' : V_{dR, \beta}^{(d)} \to \ZZ_{\ge -1}$ as follows.
For $v \in V_{dR, \beta}^{(d)}$ of the form~\eqref{eqn:v_in_basis} put:
\begin{align*}
	\ddeg' v := \max \{ \ddeg' \upomega_i : a_i \neq 0 \} \cup \{ \ddeg' \upeta_i : b_i \neq 0 \},
\end{align*}
where:
\begin{align*}
	\ddeg' \upomega_i &:= s_p(i),\\
	\ddeg' \upeta_i &:= 
	\begin{cases}
		s_p(i) - 1, & \textrm{ if  } p \nmid i,\\
		s_p(i) - 1 - d^{(1)}, & \textrm{ if } p | i \textrm{ and } i > d.
	\end{cases}
\end{align*}
\begin{Proposition} \label{prop:deg=deg'}
	For any $v \in V_{dR, \beta}^{(d)}$ we have $\ddeg v = \ddeg' v$.
\end{Proposition}

In order to describe the degree function on $V_{dR, \beta}^{(d)}$ we need an analogue of Lemma~\ref{lem:sigma_0_v_first_terms}.
\begin{Lemma} \label{lem:sigma_0_eta_first_terms}
	For any $n \ge 0$ there exist $v_1, v_2 \in V_{dR, \beta}^{(d)}$ such that $\ddeg' v_1, \ddeg' v_2 < \ddeg' \upeta_n - 1$
	and
	\begin{align*}
		\sigma_0 \upeta_n &= n \cdot \upeta_{n - 1} +
		n^{(1)} \cdot \upeta_{n - p} + v_1,\\
		\tau_0 \upeta_n &= n \cdot \beta \cdot \upeta_{n - 1} +
		n^{(1)} \cdot \beta^p \cdot \upeta_{n - p} + v_2.
	\end{align*}
\end{Lemma}
\begin{proof}
	For $p \nmid n$ the first the equalities are immediate from Lemma~\ref{lem:sigma_0_v_first_terms},
	since $\ddeg' \upeta_n = s_p(n) - 1$ and $\ddeg' \upeta_i \le s_p(i) - 1$
	for any $i$. If $p | n$ then by Lucas theorem (cf.~\eqref{eqn:lucas}):
	\begin{align*}
		\sigma_0 \upeta_n = \sum_{i < n^{(1)}} {n^{(1)} \choose i} \upeta_{p \cdot i} 
		\quad \textrm{ and } \quad
		\tau_0 \upeta_n = \sum_{i < n^{(1)}} {n^{(1)} \choose i} \beta^{n - p \cdot i} \cdot \upeta_{p \cdot i}, 
	\end{align*}
	which easily implies the desired equalities. 
\end{proof}
\begin{Corollary} \label{cor:upeta_n=sigma_upeta_n+1}
	If $p \nmid n+1$, there exist $c_1, c_2 \in k$ such that
	\[
	\upeta_n = c_1 \cdot \sigma_0 \upeta_{n+1} + c_2 \cdot \tau_0 \upeta_{n+1} + w_1
	\]
	for some $w_1$ with $\ddeg' w_1 < \ddeg' \upeta_n$. If $p | n+1$, then:
	\[
	\upeta_n = c_3 \cdot \sigma_0 \upeta_{n+p} + c_4 \cdot \tau_0 \upeta_{n+p} + w_2
	\]
	for some $c_3, c_4 \in k$ and $w_2$ with $\ddeg w_2 < \ddeg \upeta_n$.
\end{Corollary}
\begin{proof}
	Using the equalities of Lemma~\ref{lem:sigma_0_eta_first_terms}:
	\begin{align*}
		\tau_0 \upeta_{n+1} - \beta^p \cdot \sigma_0 \upeta_{n+1} = (n+1) \cdot (\beta - \beta^p) \cdot \upeta_n + v_4
	\end{align*}
	for $v_4 = v_2 - \beta^p \cdot v_1$, $\ddeg' v_4 < \ddeg' \upeta_n$. Hence one may take $c_1 := \frac{\beta^p}{(n+1) \cdot (\beta^p - \beta)}$
	and $c_2 := \frac{1}{(n+1) \cdot (\beta - \beta^p)}$. If $p | n+1$,
	the argument follows similarly by applying Lemma~\ref{lem:sigma_0_eta_first_terms}
	and defining:
	\[
		c_3 := \frac{\beta}{(n+p)^{(1)} \cdot (\beta - \beta^p)}, \quad
		c_4 := \frac{1}{(n+p)^{(1)} \cdot (\beta^p - \beta)}. \qedhere
	\]
\end{proof}
\begin{proof}[Proof of Proposition~\ref{prop:deg=deg'}]
	We prove now by induction that for every $N$:
	\[
	S_N(V_{dR, \beta}^{(d)}) = \{ v : \ddeg' v \le N \}.
	\]
	For $N = -1$ this is true. Suppose that this holds for $N-1$. Firstly, note that $\upomega_n \in S_N(V_{dR, \beta}^{(d)})$ for $s_p(n) \le N$ by Lemma~\ref{prop:homo_filtration}. Moreover:
	\begin{equation} \label{eqn:upeta_n_in_SN}
		\upeta_n \in S_N(V_{dR, \beta}^{(d)}) \textrm{ for } \ddeg' \upeta_n \le N.
	\end{equation}
	Indeed, if $p | n$ and $n < d$, then $\upeta_n = 0$ and~\eqref{eqn:upeta_n_in_SN} holds trivially.
	Otherwise, \eqref{eqn:upeta_n_in_SN} is a simple consequence of Lemma~\ref{lem:sigma_0_eta_first_terms} and the induction hypothesis.
	Therefore
	\[
		\{ v : \ddeg' v \le N \} \subset S_N(V_{dR, \beta}^{(d)}).
	\]
	Suppose now to the contrary that $v \in S_N(V_{dR, \beta}^{(d)})$ is of the form~\eqref{eqn:v_in_basis} and $M := \ddeg' v > N$. Then by Lemma~\ref{lem:sigma_0_eta_first_terms}:
	\begin{align}
		\tau_0 v &= \sum_{i \equiv -1 \pmod p} a_i (-\beta \cdot \upomega_{i-1} + i^{(1)} \cdot \beta^p \cdot \upomega_{i-p}) \nonumber\\
		&+ \sum_i b_i (i \cdot \beta \cdot \upeta_{i-1} + i^{(1)} \cdot \beta^p \cdot \upeta_{i - p}) + v_1 \nonumber\\
		&= \sum_{i \equiv -1 \pmod p} a_i (\beta \cdot \upeta_i + i^{(1)} \cdot \beta^p \cdot \upomega_{i-p}) \nonumber\\
		&+ \sum_i b_i \cdot (i \cdot \beta \cdot \upeta_{i-1} + i^{(1)} \cdot \beta^p \cdot \upeta_{i - p}) + v_1, \label{eqn:tau0v_expression}
	\end{align}
	where $\ddeg' v_1 < \ddeg' v - 1$. Fix any $n$ such that $\ddeg' \upeta_n = M$. We show now that $b_n = 0$. Assume that $p \nmid n$. On one hand, the coefficient of $\upeta_{n-1}$ in $\tau_0 v$ is zero by induction hypothesis.
	On the other hand, by~\eqref{eqn:tau0v_expression} this coefficient equals $n \cdot \beta \cdot b_n + (n + p-1)^{(1)} \cdot \beta^p \cdot b_{n + p-1}$. By considering in a similar manner the coefficient of $\upeta_{n-1}$ in $\sigma_0 v$
	we obtain the following equations:
	\begin{align*}
		n \cdot \beta \cdot b_n + (n + p-1)^{(1)} \cdot \beta^p \cdot b_{n + p-1} &= 0,\\
		n \cdot b_n + (n + p-1)^{(1)} \cdot b_{n + p-1} &= 0.
	\end{align*}
    The above equations imply that $b_n = 0$ by~\eqref{eqn:a=b=0}. 
	Suppose now that $p | n$. Then by analysing the coefficient of $\upeta_{n-p}$ in $\tau_0 v$
	and in $\sigma_0 v$ we obtain analogously:
	\begin{align*}
		b_n \cdot n^{(1)} \cdot \beta^p &+ b_{n+1 - p} \cdot \beta + a_{n-1} \cdot \beta = 0,\\
		b_n \cdot n^{(1)} &+ b_{n+1 - p} + a_{n-1} = 0.
	\end{align*}
	Again, using~\eqref{eqn:a=b=0} we obtain $b_n = 0$. Suppose now that 
	$\ddeg \upomega_n = M$ and $n \equiv -1 \pmod p$. If $n > p-1$, then by
	considering the coefficient of $\upomega_{n-p}$ in
	$\sigma_0 v$ we obtain by~\eqref{eqn:tau0v_expression} the equality $n^{(1)} \cdot a_n$,
	which yields $a_n = 0$ by induction hypothesis. If $n = p-1$, one obtains similarly $a_{p-1} = 0$
	by considering the coefficients of $\upeta_{p-1}$ in $\sigma_0 v$ and $\tau_0 v$ as above.
	Thus $\ddeg' v < M$, which yields a contradiction and ends the proof.
\end{proof}
\noindent The following equality, valid for $d < p^2 - p$, will play a crucial role in the proof of Theorem~\ref{thm:iso_of_dR_modules}':
\begin{equation} \label{eqn:VdR=S+eta}
	V_{dR, \beta}^{(d)} = S_{2p-4}(V_{dR, \beta}^{(d)}) \oplus \Span_k(\upeta_{p^2-1})
\end{equation}
This easily follows from Proposition~\ref{prop:deg=deg'} and the fact that if $s_p(n) \ge 2p - 3$ and $n < p^2$,
then $n \in \{ p^2 - p - 1, p^2 - 2, p^2 - 1 \}$.
\begin{Remark}
	The equality~\eqref{eqn:VdR=S+eta} does not
	hold for $d \ge p^2 - p$, since then $\upomega_{p^2 - p - 1} \in V_{dR, \beta}^{(d)} \setminus S_{2p-4}(V_{dR, \beta}^{(d)})$.
\end{Remark}
\noindent Let $v := \sigma_0^{p-2} \cdot \tau_0^{p-2} \cdot \upeta_{p^2 - 1}$, 
$w := \sigma_0 v$ and $\bb N(V_{dR, \beta}^{(d)}) := \Span_k(v) \oplus S_0(V_{dR, \beta}^{(d)})$. 
\begin{Lemma} \label{lem:N_module_for_dR}
	Keep the above setup and assume $d \ge p$.
	\begin{enumerate}[(1)]
		\item We have $\tau_0 v = - \beta^p \cdot w$.
		\item The element $w$ belongs to $\Span_k(\upomega_0) \setminus \{ 0 \}$.
		\item The $k[H]$-module $\bb N(V_{dR, \beta}^{(d)})$ is isomorphic to
		$V_{2, -\beta^p} \oplus k$.
	\end{enumerate}
\end{Lemma}
\begin{proof}
	(1) It suffices to prove that
	\begin{equation} \label{eqn:tau+sigma_v=0}
		(\tau_0 + \beta^p \cdot \sigma_0) \cdot v = 0.
	\end{equation}
	Indeed, let $u := (\tau_0 + \beta^p \cdot \sigma_0) \cdot \sigma_0^{p-2} \cdot \tau_0^{p-2}  \cdot \upomega_{p^2 - 1} \in V_{p^2-1, \beta}$.
	Note that $u \in S_1(V_{p^2-1, \beta})$ by Lemma~\ref{prop:homo_filtration}.
	Moreover:
	\begin{align*}
		(\tau_0 - \beta^p \cdot \sigma_0) \cdot u = (\tau_0^2 - \beta^{2p} \cdot \sigma_0^2) \cdot \sigma_0^{p-2} \cdot \tau_0^{p-2} \cdot \upomega_{p^2 - 1}
		= 0.
	\end{align*}
	But it is easy to check that
	\[
	\{ y \in S_1(V_{p^2-1, \beta}) : (\tau_0 - \beta^p \cdot \sigma_0) \cdot y = 0 \}
	= \Span_k(\upomega_0, \upomega_p).
	\]
	Thus $u \in \Span_k(\upomega_0, \upomega_p)$.
	Therefore, by passing to $V_{dR, \beta}^{(d)}$ we see that:
	\[
	(\tau_0 + \beta^p \cdot \sigma_0) \cdot v \in \Span(\upeta_0, \upeta_p) = 0,
	\]
	since $d \ge p$. This proves~\eqref{eqn:tau+sigma_v=0}.
	
	(2) Firstly, note that in $V_{p^2 - 1, \beta}$ we have $\sigma_0^{p-1} \tau_0^{p-2} \upomega_{p^2 - 1} \in S_1(V_{p^2 - 1, \beta}) = \Span_k(\upomega_0, \upomega_1, \upomega_p)$ by Proposition~\ref{prop:homo_filtration}. Hence, by passing to $V_{dR, \beta}^{(d)}$, we obtain
	\[
		w \in \Span_k(\upeta_0, \upeta_1, \upeta_p) = \Span_k(\upomega_0).
	\]
	Moreover, $w \neq 0$. Indeed, otherwise we would have $\sigma_0 v = \tau_0 v = 0$, which would easily yield contradiction with
	the equality $\ddeg \upeta_{p^2 - 1} = 2p-3$. 
	
	(3) Note that $\dim_k S_0(V_{dR, \beta}^{(d)}) = 2$. Hence the proof follows immediately from~(1) and~(2).
\end{proof}
\begin{proof}[Proof of Theorem~\ref{thm:iso_of_dR_modules}']
	Firstly, note that $d^{(1)}$ depends only on the isomorphism class of~$V_{dR, \beta}^{(d)}$. Indeed, the difference between $\dim_k S_N(V_{dR, \beta_1}^{(d)})/S_{N-1}(V_{dR, \beta_1}^{(d)})$ and
	\begin{align*}
		\# \{ j < d : s_p(j) = N \} \cup \{ j \ge d : s_p(j) - 1 = N, \quad p \nmid j \}
	\end{align*}
	equals $1$ if $N = p-1  - d^{(1)}$ (since $p^2 - p$ is the only number~$j$ divisible by~$p$ such that $s_p(j) = p-1$) and $0$ if $N > p-1  - d^{(1)}$.
	
	Suppose now that $\Phi : V_{dR, \beta_1}^{(d)} \to V_{dR, \beta_2}^{(d)}$ is an isomorphism
	of $k[H]$-modules. Again, we denote the elements of the first module by $\upomega_a$ and $\upeta_a$,
	and of the latter module by $\upomega_a'$ and $\upeta_a'$. 
	We proceed as in the proof of the second case of Theorem\ref{thm:iso_of_holo_modules}'. Observe that $\Phi$ induces an isomorphism between $\bb N(V_{dR, \beta_1}^{(d)})$ and $\bb N(V_{dR, \beta_2}^{(d)})$. Indeed, 
	$\Phi(S_0(V_{dR, \beta_1}^{(d)})) \subset S_0(V_{dR, \beta_2}^{(d)})$. Moreover,
	by~\eqref{eqn:VdR=S+eta} we have:
	\[
		\Phi(\upeta_{p^2 - 1}) = c_1 \upeta_{p^2 - 1}' + v,
	\]
	where $c_1 \in k^{\times}$ and $\ddeg v \le 2p-4$. Thus, 
	since $\sigma_0^{p-2} \tau_0^{p-2} S_{2p-4}(V_{dR, \beta_1}^{(d)}) \subset S_0(V_{dR, \beta_2})$, we obtain $\Phi(v) \in \Span_k(v') \oplus S_0(V_{dR, \beta_2}^{(d)})$.
	Hence $\bb N(V_{dR, \beta_1}^{(d)}) \cong \bb N(V_{dR, \beta_2}^{(d)})$,
	which by Lemma~\ref{lem:N_module_for_dR} yields $V_{2, -\beta_1^p} \oplus k \cong V_{2, -\beta_2^p} \oplus k$.
	We conclude that $\beta_1 = \beta_2$ by Theorem~\ref{thm:iso_of_holo_modules}'.
\end{proof}
\begin{Lemma} \label{lem:vdR_indecomposable}
	For any $0 \le d \le p^2$ the $k[G]$-module $V_{dR, \beta}^{(d)}$ is indecomposable.
\end{Lemma}
\begin{proof}
	By Lemma~\ref{lem:properties_of_v_dR}~(1) we may without loss of generality assume that $p | d$.
	Suppose that $V_{dR, \beta}^{(d)} = V_1 \oplus V_2$ for some $k[H]$-modules $V_1$, $V_2$. 
	Recall that $S_{2p-3}(V_{dR, \beta}^{(d)})/S_{2p-4}(V_{dR, \beta}^{(d)})$ is one-dimensional by~\eqref{eqn:VdR=S+eta}. This easily implies that either $S_{2p-3}(V_1)/S_{2p-4}(V_1) = 0$ or $S_{2p-3}(V_2)/S_{2p-4}(V_2) = 0$. Without loss of of generality we assume that the latter possibility holds. Then $\upeta_{p^2 - 1} + v \in V_1$ for some $v$
	such that $\ddeg v < \ddeg \upeta_{p^2 - 1}$. We prove by induction
	the following statement:
	\begin{equation} \label{eqn:upeta_n_almost_in_V1}
		\upeta_n + v_n \in V_1 \qquad \textrm{ for some } v_n \textrm{ such that }
		\ddeg v_n < \upeta_n.
	\end{equation}
	Indeed, for $n = p^2 - 1$ this is true. Suppose that this holds for all integers larger than $n$. If $p \nmid n+1$, then
	by Corollary~\ref{cor:upeta_n=sigma_upeta_n+1}:
	\[ \upeta_n = c_1 \cdot \sigma_0 \upeta_{n+1} + c_2 \cdot \tau_0 \upeta_{n+1} + w \]
	for some $w$ with $\ddeg w < \ddeg \upeta_n$ and $c_1, c_2 \in k$. By induction hypothesis,
	$\upeta_{n+1} = \xi_{n+1} - v_{n+1}$ for some $\xi_{n+1} \in V_1$. Hence:
	\begin{align*}
		\upeta_n &= c_1 \cdot \sigma_0 (\xi_{n+1} - v_{n+1}) + c_2 \cdot \tau_0 (\xi_{n+1} - v_{n+1}) + w\\
		&=\xi_n - v_n,
	\end{align*}
	where
	\begin{align*}
		\xi_n &:= c_1 \cdot \sigma_0 (\xi_{n+1}) + c_2 \cdot \tau_0 (\xi_{n+1}) \in V_1,\\
		v_n &:= c_1 \cdot \sigma_0 (v_{n+1}) + c_2 \cdot \tau_0 (v_{n+1}) - w, \qquad \ddeg v_n < \ddeg \upeta_n.
	\end{align*}
	If $p | n+1$, one proceeds in a similar manner, using the second equality of Corollary~\ref{cor:upeta_n=sigma_upeta_n+1}.
	This ends the proof of~\eqref{eqn:upeta_n_almost_in_V1}. The relation~\eqref{eqn:upeta_n_almost_in_V1} implies in particular
	that $\upeta_{d+p} \in V_1$ (since $\ddeg \upeta_{d+p} = 0$). On the other hand, by Lemma~\ref{lem:N_module_for_dR}~(2):
	\begin{align*}
		\Span_k(\upomega_0) &= \Span_k(\sigma_0^{p-1} \cdot \tau_0^{p-2} \cdot \upeta_{p^2 - 1})\\
		&= \Span_k(\sigma_0^{p-1} \cdot \tau_0^{p-2} (\upeta_{p^2 - 1} + v)) \subset V_1
	\end{align*}
	Therefore $(V_{dR, \beta}^{(d)})^H = \Span_k(\upomega_0, \upeta_{d+p}) \subset V_1^H$,
	which implies that $V_2 = 0$ by~\eqref{eqn:G-invariants_p_gp}. This ends the proof.
\end{proof}

\section{Proof of Main Theorem}
We review now the necessary facts from the articles \cite{Garnek_p_gp_covers} and \cite{Garnek_p_gp_covers_ii}.
Let $H$ be an arbitrary $p$-group.
Results of Harbater (cf. \cite{Harbater_moduli_of_p_covers}) and of Katz and Gabber (cf.~\cite{Katz_local_to_global}) imply that for any $H$-Galois algebra $\ms B$ over $k[[x]]$
there exists a unique (possibly disconnected) $H$-cover $\ms Z \to \PP^1$ ramified only over $\infty$
and such that there exists an isomorphism $\wh{\mc O}_{\ms Z, \infty} \cong \ms B$ of $k[H]$-algebras.
The cover $\ms Z \to \PP^1$ is called the \emph{Harbater--Katz--Gabber cover} associated to $(\ms B, G)$. 
Suppose now that $\pi : Z \to Y$ is an $H$-cover of smooth projective curves. Then $\wh{\mc O}_{Z, Q}$
is an $H$-Galois algebra over $\wh{\mc O}_{Y, Q} \cong k[[x]]$ for any $Q \in Y(k)$. Denote by $\ms Z_Q \to \PP^1$
the corresponding HKG-cover.\\

Suppose that $H$ and $\pi$ are as above.
Recall that a function $z \in k(Z)$ is called a \emph{magical element} for $\pi$, if $\tr_{Z/Y}(z) \neq 0$ and
for every $P \in Z(k)$ one has $\ord_P(z) \ge -d'_{Z/Y, P}$. In the sequel we use the following
properties of magical elements:
\begin{equation} \label{eqn:direct_summand}
	\begin{minipage}{0.9 \textwidth}
		Suppose that $\pi$ has a magical element and that all the inertia subgroups of $\pi$ are normal in $H$. Then
		for every $P \in Z(k)$ the $k[H]$-module $\Ind^H_{H_P} H^0(\Omega_{\ms Z_{\pi(P)}})$ is a direct summand
		of $H^0(\Omega_Z)$ and $\Ind^H_{H_P} H^1_{dR}(\ms Z_{\pi(P)})$ is a direct summand
		of $H^1_{dR}(Z)$ (cf. \cite[Theorem~1.1]{Garnek_p_gp_covers} and \cite[Theorem~1.2]{Garnek_p_gp_covers_ii}).
	\end{minipage}
\end{equation}
\begin{equation} \label{eqn:magical_tower}
	\begin{minipage}{0.9 \textwidth}
		If $\pi$ factors through a Galois cover $Z' \to Y$ and both the covers $Z \to Z'$ and $Z' \to Y$
		have magical elements, then $\pi$ has a magical element as well (cf. \cite[Lemma~8.1]{Garnek_p_gp_covers}).
	\end{minipage}
\end{equation}
\begin{equation} \label{eqn:magical_Zp_covers}
	\begin{minipage}{0.9 \textwidth}
		If $H = \ZZ/p$ and there exists a point $P \in Z(k)$ with ramification jump at least $2g_Y \cdot p$,
		then $\pi$ has a magical element (cf. \cite[Lemma~7.2 and~7.3]{Garnek_p_gp_covers}).
	\end{minipage}
\end{equation}
\begin{Remark} \label{rem:Indec_HKG}
	We conjecture that the assumptions (A) and (B) of \cite[Theorem~1.2]{Garnek_p_gp_covers_ii} can be replaced by the following
	weaker assumption:
	\begin{center}
		The cover $\pi : Z \to Z/H$ has no \'{e}tale subcovers.
	\end{center}
	In particular, the study of $\Indec^{dR}(k[H])$ and of $\Indec^{Hdg}(k[H])$ for a $p$-group $H$ should essentially come down to the study
	of the possible indecomposable summands of the Harbater--Katz--Gabber covers and of \'{e}tale covers.
\end{Remark}
\noindent We show now how to prove Main Theorem assuming the following result.
\begin{Proposition} \label{prop:cover_extending_given}
	Let $H$ be a $p$-subgroup of a finite group $G$.
	Fix an $H$-cover $X \to \PP^1$ ramified only at $\infty$. 
	There exists a connected $G$-cover $Z \to \PP^1$ such that 
	$\wh{\mc O}_{Z, \infty} \cong \Ind^G_H \wh{\mc O}_{X, \infty}$ as $k[G]$-algebras
	and that the $H$-cover $Z \to Z/H$ has a magical element.
\end{Proposition}
\begin{proof}[Proof of Main Theorem]
	We keep the setup of Main Theorem. Note that since the $p$-Sylow subgroup
	of $G$ is not cyclic, it contains a subgroup $H$ isomorphic to $\ZZ/p \times \ZZ/p$ (cf. \cite[exercise~5.5.20]{DummitFoote2004}). Consider the sets:
	\begin{alignat*}{2}
		\mc A^{\star} &:= \{ M \in \Indec(k[H]) : M \textrm{ is a direct summand of } H^1_{\star}(X)|_H\\
		&\, \textrm{ for a smooth projective curve } X/k \textrm{ with an action of } G \}
	\end{alignat*}
	for $\star \in \{ Hdg, dR \}$. Note that
	\[
		\mc A^{\star} = \bigcup_{M \in \Indec^{\star}(k[G])} \{ N \in \Indec(k[H]) : N \textrm{ is a direct summand of } M|_H  \}.
	\]
	Therefore it suffice to prove that the sets $\mc A^{Hdg}, \mc A^{dR}$ are infinite.
	Pick any $\beta \in k \setminus \FF_p$ and $1 < d < p^2 - 1$. 
	Let $Z \to \PP^1$ be the $G$-cover obtained by Proposition~\ref{prop:cover_extending_given}
	for $H$, $G$ as above and $X := X_{p^2 + 1, \alpha}$, where $\alpha := - \beta^{-p}$.
	Note that $\infty \in \PP^1(k)$ has $[G:H]$ points in the preimage under $Z \to \PP^1$.
	The inertia subgroups of those points are subgroups conjugated to $H$. Pick any point $P$ with inertia subgroup equal to $H$. Let $Q \in (Z/H)(k)$ be its image under the map $Z \to Z/H$. Then $\wh{\mc O}_{Z, Q} \cong \wh{\mc O}_{X, \infty}$ and $\ms Z_Q \to \PP^1$ (the HKG-cover approximating the cover $Z \to Z/H$ over $Q$) is isomorphic to the cover $X \to \PP^1$.
	
	Moreover, since $Z \to Z/H$ has a magical element, by~\eqref{eqn:direct_summand}
	$H^0(\Omega_X)$ is a direct summand of the $k[H]$-module $H^0(\Omega_Z)|_H$. Therefore $V_{d, \beta}$ (resp. $V_{dR, \beta}^{(d)}$) is
	a direct summand of $H^0(\Omega_Z)$ (resp. $H^1_{dR}(Z)$) and $V_{d, \beta} \in \mc A^{Hdg}$ (resp. $V_{dR, \beta}^{(d)} \in \mc A^{dR}$).
	Thus, by Theorems~\ref{thm:iso_of_holo_modules}' and~\ref{thm:iso_of_dR_modules}', the sets $\mc A^{Hdg}$ and $\mc A^{dR}$ are infinite.
\end{proof}
The rest of this section will be occupied by the proof of Proposition~\ref{prop:cover_extending_given}. The idea of the proof is to consider three
$G$-covers:
\begin{itemize}
	\item the first one, $Z_1 \to \PP^1$, will be a (connected) $G$-cover,
	
	\item the second one, $Z_2 \to \PP^1$, will be a disconnected $G$-cover that will ensure the existence of the magical element,
	
	\item the third one, $Z_3$, will be a disconnected $G$-cover
	that consists of disjoint union of copies of $X$, the given cover.
\end{itemize}
Then we glue them using the following consequence of patching theory (see \cite[Theorem~5]{Harbater_Stevenson_patching_thickening} and \cite[Lemma~3.3]{Das_Kumar_inertia_alternating} for a similar reasoning).
\begin{Lemma} \label{ref:glueing_lemma}
	Let $Z_{\circ} \to Y_{\circ}$ be a $G$-cover of connected (but possibly reducible) nodal curves over $k$ with branch locus $B_{\circ} \subset Y^{\circ}(k)$. Suppose that
	the $B_{\circ}$ is contained in the smooth locus of $Y_{\circ}$. Then there exists a connected $G$-cover $Z \to Y$ of smooth projective curves
	$Z \to Y$ with branch locus $B$ such that there exists a bijection $\theta : B_{\circ} \to B$ satisfying for every $b \in B_{\circ}$
	\begin{equation} \label{eqn:z_*_and_z_0_stalks2}
		\wh{\mc O}_{Z_{\circ}, b} \cong \wh{\mc O}_{Z, \theta(b)} \qquad \textrm{ as $k[G]$-modules.}
	\end{equation}
\end{Lemma}
\begin{proof}
	Let $R \subset Z_{\circ}(k)$ be the ramification locus of $Z_{\circ} \to Y_{\circ}$.
	The base change of $Z_{\circ} \setminus R \to Y_{\circ} \setminus B$ to $k[[t]]$ along with the trivial deformation of complete local rings $\wh{\mc O}_{Z, b} \otimes_k k[[t]]$ for $b \in B_{\circ}$ defines a thickening problem in the sense of \cite{Harbater_Stevenson_patching_thickening}. By \cite[Theorem~4]{Harbater_Stevenson_patching_thickening} there exists
	a $G$-cover $Z_* \to Y_*$ of smooth curves over $k[[t]]$ with special fiber $Z_{\circ} \to Y_{\circ}$
	such that the following conditions hold:
	\begin{itemize}
		\item the branch locus of $Z_* \to Y_*$ is $B \times_k k[[t]]$,
		
		\item for any $b \in B_{\circ} \times_k k[[t]]$ there exists a $G$-equivariant isomorphism
		\begin{equation} \label{eqn:z_*_and_z_0_stalks}
			\wh{\mc O}_{Z_*, b} \cong \wh{\mc O}_{Z_{\circ}, b} \otimes_k k[[t]].
		\end{equation}
	\end{itemize}
	Since the structure is of finite type, there is a $k[t]$-subalgebra $T \subset k[[t]]$
	over which the covering morphism, $G$-Galois action, branch
	points are defined and such that the structure sheaf of the base space is
	locally free over~$T$. Moreover, since ~\eqref{eqn:z_*_and_z_0_stalks} are isomorphisms of $\wh{\mc O}_{Y_{\circ}, b} \otimes_k k[[t]]$-algebras of finite dimension, they are
	defined by matrices involving finitely many functions in $\wh{\mc O}_{Y_0, b} \otimes_k k[[t]]$. Therefore by possibly enlarging $T$, we may also assume that the isomorphisms ~\eqref{eqn:z_*_and_z_0_stalks} are defined over $T$.
	
	Thus we obtain a cover $Z_T \to Y_T$ of curves over $T$, which induces $Z_{\circ} \to Y_{\circ}$
	under the pullback $\Spec k[[t]] \to \Spec T$. Analogously, by possibly enlarging $T$, we may assume that $Z_T$ and $Y_T$ are smooth over $T$ (since the generic fibers of $Y_T$ and $Z_T$ are smooth). Then the pullback of $Z_T \to Y_T$ through any point $\xi \in T(k)$ yields the desired cover $Z \to Y$.
\end{proof}
The following lemma will allow us to construct the cover $Z_2 \to \PP^1$ that will ensure the existence of the magical element.
\begin{Lemma} \label{lem:big_jumps}
	Let $H$ be a finite $p$-group. Fix numbers $c_1, c_2 \in \RR$. Denote by $\{ e \} = H_0 \unlhd H_1 \unlhd \ldots \unlhd H_q = H$ the central series of $H$. There exists a connected $H$-cover $X' \to \PP^1$ ramified only at infinity such that for every $i = 1, 2, \ldots, q$ the $\ZZ/p$-cover
	$X'/H_{i-1} \to X'/H_i$ has lower ramification jump $m_i$ above $\infty$, where
	\begin{equation} \label{eqn:jump_is_big}
		m_i > c_1 \cdot g_{X'/H_i} + c_2.
	\end{equation}
\end{Lemma}
\begin{proof}
	The cover $X' \to \PP^1$ can be easily constructed inductively as
	in~\cite[proof of Theorem~1.5]{Garnek_p_gp_covers}. Indeed, suppose that
	$X^{(i)} \to \PP^1$ is an $H/H_i$-cover of $\PP^1$
	ramified only at~$\infty$ such that the ramification jumps of the intermediate $\ZZ/p$-covers satisfy the inequalities~\eqref{eqn:jump_is_big}.
	Let $\mc E$ denote the set of extensions of $X^{(i)} \to \PP^1$ to an $H/H_i$-cover
	ramified only above $\infty$. Write also $M_{\AA^1, \ZZ/p}$ for the set of $\ZZ/p$-covers of $\PP^1$ ramified only above $\infty$. Note that $M_{\AA^1, \ZZ/p}$ has a natural group structure. With this notation, it can be proven that $\mc E \neq \varnothing$ and that there is a natural action
	\[
		M_{\AA^1, \ZZ/p} \times \mc E \to \mc E, \qquad
		(C, \mc X) \mapsto \mc X_C, 
	\] 	
	see~\cite[Section~8.1]{Garnek_p_gp_covers} and~\cite{Harbater_moduli_of_p_covers}. Fix $\mc X \in \mc E$.
	One can show that if the ramification jump of $C \in M_{\AA^1, \ZZ/p}$ 
	at $\infty$ is big enough, then $X^{(i-1)} := \mc X_C$ satisfies~\eqref{eqn:jump_is_big},
	see \cite[Lemma~8.2]{Garnek_p_gp_covers}. Finally, we take $X' := X^{(0)}$.
\end{proof}
\begin{proof}[Proof of Proposition~\ref{prop:cover_extending_given}]
	Denote by $\{ e \} = H_0 \unlhd H_1 \unlhd \ldots \unlhd H_q = H$ the central series of~$H$.
	Let $\pi_1 : Z_1 \to \PP^1$ be any connected $G$-cover with branch locus $B_1$, where $0, 1 \not \in B_1$. Let $X' \to \PP^1$ be a cover satisfying the conditions of Lemma~\ref{lem:big_jumps}
	for
	\begin{equation*}
		c_1 := 2p \cdot \# G, \qquad c_2 := 2p \cdot \# G \cdot (g_{Z_1} +  \cdot g_X).
	\end{equation*}
	Consider the following cover $Z_{\circ} \to Y_{\circ}$ of nodal curves:
	\begin{itemize}
		\item $Y_{\circ}$ is constructed by taking three copies $Y_1, Y_2, Y_3$ of $\PP^1$ and identifying $0 \in Y_1$ with $0 \in Y_2$ and
		$1 \in Y_1$ with $1 \in Y_3$ in ordinary double points,
		
		\item $Z_{\circ}$ is constructed by taking $Z_1$, $Z_2 := \Ind^G_H X'$, $Z_2 := \Ind^G_H X$ and identifying
		the fibers of $0$ in $Z_1$ and $Z_2$ and the fibers of $1$ in $Z_1$ and $Z_3$ in ordinary double points.
	\end{itemize}
	Denote by $\infty_i$ the point corresponding to $\infty$ on $Y_i$
	for $i = 2, 3$. Let $B_{\circ} := B_1 \cup \{ \infty_2, \infty_3 \}$. Let $Z \to Y$ be the $G$-cover obtained by Lemma~\ref{ref:glueing_lemma}
	with the branch locus $B$ and bijection $\theta : B_{\circ} \to B$.
	Note that by flatness $g_Y = g_{Y^{\circ}} = 0$, which yields $Y \cong \PP^1$. Without loss of generality we may assume that $\theta(\infty_3) = \infty \in \PP^1(k)$.\\
	
	Finally, we show that the cover $Z \to Z/H$ has a magical element.
	Write $Z^{(i)} := Z/H_i$ for any $i = 0, \ldots, q$. By~\eqref{eqn:magical_tower} it suffices to show that
	for every $i = 1, \ldots, q$ the cover $Z^{(i-1)} \to Z^{(i)}$ has a magical
	element. To this end, we start by estimating the genus of $Z^{(i)}$.
	Note that the ramification locus of the cover $Z^{(i)} \to \PP^1$
	is of the form $R := R_1 \cup R_2 \cup R_3$, where $R_1$ is the preimage of $\theta(B_1)$
	under $Z^{(i)} \to \PP^1$ and $R_j$ is the preimage of $\theta(\infty_j)$ under the same map for $j = 2, 3$. In particular, $\# R_j = [G:H]$. Therefore,
	by~\eqref{eqn:z_*_and_z_0_stalks2}
	and by the Riemann--Hurwitz formula for the cover $Z_1 \to \PP^1$ we have
	\begin{align*}
		\sum_{P \in R_1} d_{Z^{(i)}/\PP^1, P} \le 2(g_{Z_1} - 2) + 2 \# G < \# G \cdot 2g_{Z_1}.	
	\end{align*}
	Similarly, for any $P_2 \in R_2$,
	$P_3 \in R_3$ we have $d_{Z^{(i)}/\PP^1, P_2} < [H : H_i] \cdot 2g_{X'}$ and $d_{Z^{(i)}/\PP^1, P_3} < [H : H_i] \cdot 2g_{X/H_i}$.
	Therefore, by Riemann--Hurwitz formula for the morphism $Z^{(i)} \to \PP^1$
	\begin{align*}
		2(g_{Z^{(i)}} - 1) &= 2 \cdot [G : H_i] \cdot (0 - 1) +
		\sum_{P \in R} d_{Z^{(i)}/\PP^1, P}\\
		&<  \# G \cdot 2g_{Z_1} + [G:H] \cdot [H : H_i] \cdot (2g_{X'}+2g_{X/H_i}).
	\end{align*}
	Therefore, by~\eqref{eqn:z_*_and_z_0_stalks2}, Lemma~\ref{lem:big_jumps} and the definition of $c_1$ and $c_2$, the cover $Z^{(i-1)} \to Z^{(i)}$ has a point with
	the ramification jump at least
	\begin{align*}
		2p \cdot \# G \cdot g_{X'/H_i} + 2p \cdot \# G \cdot g_{Z_1} + 2p \cdot \# G \cdot g_X \ge 2 p \cdot g_{Z^{(i)}}.
	\end{align*}
	Thus by~\eqref{eqn:magical_Zp_covers}, the cover $Z^{(i-1)} \to Z^{(i)}$ has a magical element. This ends the proof.
\end{proof}
\bibliography{bibliografia}
\end{document}